\documentclass[oneside]{amsart}
\usepackage[utf8]{inputenc}

\usepackage{euscript}
\usepackage[T2A]{fontenc}
\usepackage[utf8]{inputenc}
\usepackage[normalem]{ulem}

\usepackage[english]{babel} 
\usepackage{amsmath}
\usepackage{mathtools}
\usepackage{amsfonts}
\usepackage{amssymb, amsthm}
\usepackage{dsfont}

\usepackage{hyperref}

\usepackage{color}

\usepackage{amsmath}
\usepackage{subfig}
\usepackage{geometry}
\usepackage{faktor}

\usepackage{multicol} %колонки

\usepackage{graphicx}
\graphicspath{}
\DeclareGraphicsExtensions{.pdf,.png,.jpg} 

\newcommand{\mbbN}{\mathbb{N}}
\newcommand{\mbbZ}{\mathbb{Z}}
\newcommand{\mbbP}{\mathbb{P}}

\newcommand{\mbbH}{\mathbb{H}}
\newcommand{\mbbE}{\mathbb{E}}

\newcommand{\indicator}{\mathds{1}}

\newcommand{\eps}{\varepsilon}
\newcommand{\del}{\delta}

\newcommand{\intl}{\int \limits}
\newcommand{\suml}{\sum \limits}
\newcommand{\liml}{\lim \limits}

\newcommand{\supl}{\sup \limits}
\newcommand{\maxl}{\max \limits}

\newcommand{\p}{\prime}

\newcommand{\abs}[1]{{\left| #1 \right|}}

\newcommand{\wpart}[1]{{\left[ #1 \right]}}

\DeclareMathOperator{\diam}{diam}

\newcommand{\dint}{{\text{d}}}

\renewcommand{\le}{\leqslant}
\renewcommand{\ge}{\geqslant}

\theoremstyle{plain}
\newtheorem{thm}{Theorem}%Theorem} %[section], чтобы нумеровать сначала в каждом разделе
\newtheorem*{thm*}{Theorem}%Theorem}
\newtheorem*{lm*}{Lemma}%Lemma}
\newtheorem{lm}{Lemma}%Lemma}
\theoremstyle{definition}
\newtheorem{defn}{Definition}%Definition}
\newtheorem{qst}{Question}%Question}

\theoremstyle{remark}
\newtheorem*{rem}{Remark}%Remark}

\newcommand{\lr}[1]{{\left( #1\right)}}

\DeclareMathOperator{\Ind}{CInd}
\DeclareMathOperator{\spn}{spn}
\DeclareMathOperator{\sep}{sep}

\newcommand{\acts}[1]{\stackrel{{{{#1}}}}{{\curvearrowright}}}

\begin{document}
%\selectlanguage{English}

\title[Non-existence of a universal zero entropy system]{Non-existence of a universal zero entropy system for non-periodic amenable group actions}
\author{Georgii Veprev}
\date{\today}
\thanks{The work is supported by Ministry of Science and Higher Education of the Russian Federation, agreement №~075-15-2019-1619. 
The work is also supported by the V.~A.~Rokhlin scholarship for young mathematicians.
}
\address{Leonhard Euler International Mathematical Institute in St. Petersburg, 
\newline 14th Line 29B, Vasilyevsky Island, St. Petersburg, 199178, Russia}
\email{egor.veprev@mail.ru}

\maketitle

\begin{abstract}
    Let $G$ be a non--periodic amenable group. We prove that there does not exist a topological action of $G$ for which the set of {ergodic} invariant measures coincides with the set of all {ergodic} measure--theoretic $G$--systems of entropy zero. Previously J.~Serafin, answering a question by B.~Weiss, proved the same for $G = \mbbZ$.
    %The well--known variational principle leads to a question about the existence of a topological dynamical system in which all measure--theoretic zero entropy systems and only them can be represented. 
    %This question originally goes back to B.~Weiss. In the work by J.~Serafin, it was proved that such a system does not exist in the case of one automorphism (actions of~$\mbbZ$).
    %This question originally posed by B.~Weiss was negatively answered in the paper by~J.~Serafin for the case of one automorphism (actions of~$\mbbZ$).
    %We prove the non--existence of a universal zero entropy system for all non--periodic amenable groups.
    %This question originally goes back to B.~Weiss and 
    %{В работе Я.~Серафина обсуждается восходящий к Б.~Вейсу вопрос о существовании топологической динамической системы, в которой могут быть реализованы все метрические системы нулевой энтропии и только они. Вариационный принцип гарантирует, что такая система должна иметь нулевую топологическую энтропию. В работе Я. Серафина доказано, что такой системы не существует в случае одного преобразования (действий группы~$\mathbb{Z}$). Мы доказываем несущестование универсальной системы нулевой энтропии для действий любой непериодической аменабельной группы.}
\end{abstract}

\section{Introduction}%Introduction}

In this work, we generalize the result of \cite{S} to the case of a non--periodic amenable group $G$. Namely, we prove that there does not exist a topological zero entropy system which is universal for {ergodic} measure--theoretic $G$--actions with zero entropy. The main tool that we implement in the proof is the notion of \emph{scaling entropy} which was proposed in works by A.~M.~Vershik \cite{V1,V3}. The theory of scaling entropy was developed in \cite{PZ, VPZ, Z1, Z2}. In this work, we prove the lower bound for the $\eps$--entropy of an averaging of independent metrics (Lemma~\ref{lm_estimate}) which allows us to estimate a scaling entropy for a special series of examples of $G$--actions. The existence of such a series (see Definition~\ref{def_almost}) implies the absence of a universal zero entropy system.  

{
%As a consequence of our arguments, we answer the question posed in~\cite{KKW}. 
The concrete examples of dynamical systems that we use to prove our main result (Theorem~\ref{thm_univ}) also give the answer to the question mentioned in~\cite{KKW}. 
That is, we show in Theorem~\ref{thm_kush} that for any sequence $A = \{a_n\}_{n=1}^\infty$ such that $a_{n+1} - a_n \to \infty$ there is a measure--preserving transformation with zero entropy  which has positive \emph{sequential entropy}~\cite{Kush} with respect to $A$.}
%{В данной работе обобщается результат статьи \cite{S} на случай непериодической аменабельной группы $G$. А именно, мы доказываем, что не существует системы нулевой топологической энтропии универсальной для действий группы $G$ нулевой метрической энтропии. Основным инструментом доказательства служит понятие масштабированной энтропии, введенное в работах А.~М.~Вершика~\cite{V1,V3}. Теория масштабированной энтропии получила развитие в работах \cite{VPZ, Z1, PZ, Z2}. В настоящей работе мы доказываем {нижнюю оценку на энтропию усреднения метрик} (лемма~\ref{lm_estimate}), которая позволяет вычислить масштабирующую энтропию для серии примеров действий группы $G$. Существование такой (см. определение \ref{def_almost}) серии примеров влечёт отсутствие универсальной топологической системы нулевой энтропии.}

\subsection{Classical notions}
Let us recall several basic notions of the entropy theory of dynamical systems (see, e.\,g., \cite{KL}). A countable group $G$ is called \emph{amenable} if it satisfies the F\o lner condition, meaning that there is a sequence $\{F_n\}_{n=1}^\infty$ of finite subsets of $G$ such that for any $g \in G$
%Напомним классические понятия энтропийной теории динамических систем (см., например,~\cite{KL}). Счетная группа $G$ называется \emph{аменабельной} если она удовлетворяет условию Фёльнера, то есть, существует такая последовательность конечных подмножеств $F_n \subset G$,  что для любого $g \in G$
\[
\liml_{n \to \infty} \frac{\abs{gF_n \triangle F_n}}{\abs{F_n}} = 0.
\]
In this case, the sequence is called a left F\o lner sequence. A right F\o lner sequence can be defined in a similar way. We will consider left actions of amenable groups on Lebesgue measure spaces without point masses, i.\,e., measure spaces  that are isomorphic to the unit segment with the Lebesgue measure.
%Такую последовательность $F_n$ мы будем называть левой последовательностью Фёльнера. Аналогично определяется правая последовательность Фельнера. Мы будем рассматривать левые действия группы $G$ на пространствах Лебега, не содержащих атомов, то есть, на пространствах, изоморфных единичному отрезку $[0,1]$ с мерой Лебега. 

\subsubsection{Amenable Topological Entropy}
Let an amenable group $G$ act by homeomorphisms on a compact metric space $(X, d)$. The amenable topological entropy of this action can be defined in the following way: 
%Пусть аменабельная группа $G$ действует гомеоморфизмами на метрическом компакте $(X, d)$. Топологическая энтропия действия определяется следующим образом: 
\[
h_{top}(X, G) = 
\supl_{\eps > 0} \limsup\limits_{n\to +\infty} \frac{1}{\abs{F_n}}\log \spn(d, F_n, \eps) =
\supl_{\eps > 0} \limsup\limits_{n\to +\infty} \frac{1}{\abs{F_n}}\log \sep(d, F_n, \eps),
\]
where $\spn(d, F_n, \eps)$ and $\sep(d, F_n, \eps)$ are cardinalities of the minimal $\eps$--net and the maximal $\eps$--separated set respectively for the maximized metric
%где $\spn(d, F_n, \eps)$ и $\sep(d, F_n, \eps)$ --- размер минимальной $\eps$--сети и максимального $\eps$--разделенного множества соответственно для метрики 
\[
G_{max}^n d (x,y) = \maxl_{g \in F_n} d(gx, gy), \qquad x,y \in X.
\]   
%Как показано в \cite{?}, 
The value of $h_{top}(X,G)$ does not depend on the choice of F\o lner sequence and forms a topological invariant of a dynamical system.
%Величина $h_{top}(X,G)$ не зависит от выбора фёльнеровской последовательности множеств $F_n$ и является инвариантом топологической динамической системы. 

\subsubsection{Amenable Measure Entropy}

Assume that $G$ acts by automorphisms on a standard probabi\-lity space $(X, \mu)$. The \emph{entropy} of a  measurable partition $\xi$ is defined as the following non--negative value:
%Предположим, что группа $G$ действует автоморфизмами на стандартном вероятностном пространстве $(X, \mu)$. Для измеримого разбиения $\xi$ пространства $(X,\mu)$ символом $H(\xi)$ обозначим его \emph{энтропию}, т.\,е. следующую неотрицательную величину:
\[
H(\xi) = -\intl_X \log\mu(\xi(x))\ \dint \mu(x),
\]
where $\xi(x)$ stands for the cell of $\xi$ that contains a point $x \in X$. For a partition $\xi$ with finite entropy, define its entropy with respect to the measure--preserving action of $G$:
%где символом $\xi(x)$ обозначен элемент разбиения $\xi$, содержащий точку $x \in X$. Далее, для измеримого разбиения $\xi$ с конечной энтропией $H(\xi)$ определим его энтропию относительно действия группы $G$:
\vspace{-5pt}
$$
h(\xi) = \liml_{n\to+\infty} \frac{1}{\abs{F_n}}H\left(\bigvee\limits_{g \in F_n} g^{-1}\xi\right),
$$
where $\vee$ is the refinement sign. \emph{Amenable measure entropy} of the action is defined by
%{где символом $\vee$ обозначено произведение разбиений}. \emph{Метрическая энтропия действия группы}~$G$ определяется следующим образом:
\[
h(X,\mu, G) = \sup\{h(\xi)\colon H(\xi)<+\infty\}.
\]
Amenable measure entropy is independent of the choice of F\o lner sequence and forms an invariant of a measure--preserving system.
%Метрическая энтропия не зависит от выбора фёльнеровской последовательности $F_n$ и является инвариантом метрической динамической системы.

\subsubsection{The Variational Principle}
The variational principle is a well--known relation between topological and measure--theoretic entropies. Let $G \acts{} (X,d)$ be a continuous action of $G$ on a compact metric space and $M_G(X)$ be a set of all $G$--invariant Borel probability measures on $X$. Then the following holds 
%Хорошо известен вариационный принцип для действия аменабельной группы $G$. А именно, пусть $M_G(X)$ есть множество $G$--инвариантных вероятностных борелевских мер на компактном метрическом пространстве $X$. Тогда
\[
h_{top}(X,G) = \sup_{\mu\in M_G(X)} h(X, \mu, G),
\]
%в частности, множество $M_G(X)$ не пусто. Отметим также, что $h_{top}(X, G) \ge h(X,\mu, G)$, в частности, если топологическая энтропия равна нулю, то и метрическая тоже.
in particular, $M_G(X)$ is nonempty. Note that $h_{top}(X, G) \ge h(X,\mu, G)$, therefore, if the topolo\-gical entropy is zero, then the measure--theoretic entropy is zero as well. {Moreover, if all ergodic measures in $M_G(X)$ have entropy zero, then the topological entropy is also zero}.

\subsection{Universal systems of entropy zero}%Universal system of entropy zero}
Questions about the existence of universal dynamical systems in various senses have been studied, for example, in~\cite{DS, S, SW, VZ2, W}.  
We will use the following definition.
%proposed in \cite{S}. 
\begin{defn}\label{def_univ}
    A topological system~$(X, G)$ is called \emph{universal} for some class~$\mathcal{S}$ of {ergodic} measure--preserving actions of~$G$ if the following two conditions hold. 
    \begin{enumerate}
        \item For any {ergodic} $\mu \in M_G(X)$ the system $(X, \mu, G)$ belongs to $\mathcal{S}$.
        \item For any $(Y,\nu, G) \in \mathcal{S}$ there exists an invariant measure $\mu$ on $(X, G)$ such that  $(X, \mu, G)$ is measure--theoretically isomorphic to $(Y,\nu, G)$.
     \end{enumerate}
\end{defn}
%В работе~\cite{S} предложено следующее определение, удобное в контексте данной работы. 
%\begin{defn}\label{def_univ}
%    Топологическая система $(X, G)$ называется \emph{универсальной} для класса $\mathcal{S}$ метрических действий группы $G$, если %выполняются следующие два условия.
%    \begin{enumerate}
%        \item Для любой $\mu \in M_G(X)$ система $(X, \mu, G)$ принадлежит классу $\mathcal{S}$.
%        \item Для любой системы $(Y,\nu, G) \in \mathcal{S}$ существует инвариантная мера $\mu$ на $(X, G)$, такая, что метрическая система $(X, \mu, G)$ изоморфна $(Y,\nu, G)$.
%    \end{enumerate}
%\end{defn}

In \cite{S}, the question about the existence of a universal system for all zero entropy systems appears. This question goes back to B.~Weiss. Due to the variational principle and the first condition in Definition~\ref{def_univ}, such a system must have zero topological entropy.
%В работе~\cite{S} обсуждается восходящий к Б. Вейссу вопрос о существовании универсальной топологической системы для класса $\mathcal{S}$, состоящего из систем нулевой энтропии. В силу вариационного принципа и первого свойства из определения~\ref{def_univ}, такая топологическая система должна обладать нулевой топологической энтропией.

\begin{qst}\label{quest_zeroentropy}
    Does there exist a system $(X, G)$ with zero topological entropy which is universal for the class of all {ergodic} measure--preserving actions of zero entropy?
    \footnote{Note that the notion of a "universal" system is often used in a slightly different sense. Sometimes it is only required to satisfy the second condition in Definition~\ref{def_univ}. In this case, the question can be easily solved via the famous Krieger's finite generator theorem (see \cite{Kr}): every {ergodic} automorphism $T$ with entropy less than one can be realized in the left shift on $\{0,1\}^\mathbb{Z}$.}
\end{qst}
%\begin{qst}\label{quest_zeroentropy}
    %Существует ли система $(X, G)$ нулевой топологической энтропии, универсальная для класса всех действий нулевой метрической энтропии?\footnote{Отметим, что часто термин ``универсальная'' употребляется в чуть ином смысле --- требуется выполнение лишь второго условия из определения~\ref{def_univ}. В таком случае вопрос становится не столь содержательным --- сдвиг на $\{0,1\}^\mathbb{Z}$ допускает реализацию любого автоморфизма $T$ с метрической энтропией $h(T)$ меньшей 1, это следует из теоремы Кригера об образующей \cite{Kr}. Теорема Кригера справедлива для действий аменабельных групп, а также более общего класса счетных групп. }
%\end{qst}

The work by J.~Serafin \cite{S} gives the negative answer to Question~\ref{quest_zeroentropy} for the case of the group $\mbbZ$. However, this question is still open for general amenable groups. The approach of \cite{S} is based on the notions of symbolic and measure--theoretic complexity of a dynamical system (see also~\cite{F}) and special constructions of systems with rapidly growing measure--theoretic complexity. The author of that work points out that this approach did not work for the case of amenable groups due to insufficient development of the theory of symbolic extensions. Let us remark that the notion of measure--theoretic complexity is closely related to the notion of scaling entropy that we use. The main result of our work is the following theorem, which gives the negative answer to Question~\ref{quest_zeroentropy} in the case of a  non--periodic amenable group $G$.

%В работе \cite{S} дан отрицательный ответ на вопрос~\ref{quest_zeroentropy} для действия группы $\mbbZ$, однако вопрос для действия аменабельных групп остается открытым. Подход работы~\cite{S} основан на понятии символической и метрической сложности динамической системы (см. также~\cite{F}), а также специальных конструкциях динамических систем с промежуточным ростом метрической сложности. Автор указывает, что данный подход для аменабельных групп не дал искомого результата ввиду недостаточной степени развития теории символических продолжений. Стоит отметить, что понятие метрической сложности тесно связано с понятием масштабированной энтропии, используемой нами. Основным результатом данной работы является следующая теорема, дающая отрицательный ответ на поставленный вопрос~\ref{quest_zeroentropy} для случая непериодической аменабельной группы $G$.

\begin{thm}\label{thm_univ}
    Let $G \acts{} (X,d)$ be a continuous action of a countable non--periodic amenable group~$G$ on a compact metric space $(X, d)$. Suppose that for any {ergodic} measure--preserving dynamical system~$(Y, \nu, G)$ with zero entropy there exists an invariant measure~$\mu$ on~$X$ with   
    \[
    (X, \mu, G) \cong (Y, \nu, G).
    \]
    Then the topological entropy of $(X,d,G)$ is positive.
\end{thm}

%\begin{thm}\label{thm_univ}
%Пусть счетная непериодическая аменабельная группа $G$ действует гомеоморфизмами на метрическом компакте $(X,d)$. Предположим, что для любой динамической системы $(Y, \nu, G)$ нулевой метрической энтропии существует мера $\mu$ на~$X$, инвариантная относительно $G$, такая, что
%    \[
%    (X, \mu, G) \cong (Y, \nu, G).
%    \]
%    Тогда топологическая энтропия системы $(X,d,G)$ положительна.
%\end{thm}

% \begin{rem}
% Утверждение теоремы остается верным для групп, содержащих в качестве подгруппы $\oplus \mbbZ_2$.
% \end{rem}

%Автор благодарен В.~В.~Рыжикову, который привлек его внимание к данному вопросу.

\section{Scaling entropy} 

\subsection{Epsilon--entropy and scaling entropy sequence}
%\subsection{Эпсилон--энтропия и масштабирующая энтропийная последовательность}

The main tool in the proof of Theorem~\ref{thm_univ} is the notion of scaling entropy introduced by A.~Vershik in~\cite{V2010, V1,V3}. The closely related notion of measure--theoretic complexity appears in the works by S.~Ferenczi~\cite{F}, A.~Katok and J.-P.~Thouvenot~\cite{KT} and uses symbolic encoding and Hamming metrics. Vershik's approach is based on the dynamics of functions of several variables, namely \emph{admissible semimetrics}\footnote{Occasionally the term "quasimetric" is used instead of "semimetric".} (see~\cite{VPZ} for details). The theory of scaling entropy was developed by A.~Vershik, F.~Petrov, and P.~Zatitskiy  in~\cite{PZ, VPZ, Z1, Z2}. Let us recall the basic concepts and statements of this theory.   
%Основным инструментом в доказательстве теоремы~\ref{thm_univ} служит понятия масштабированной энтропии, введенное  в работах Вершика~\cite{V1,V3}. В работах Ференци \cite{F} и Катка--Тувено \cite{KT} рассматривалось {близкое} понятие метрической сложности динамической системы, использующее символическое кодирование и метрики Хэмминга. Предложенный Вершиком подход основывается на динамике функций нескольких переменных, а именно, \emph{допустимых} полуметрик (см.~\cite{VPZ}). Теория масштабированной энтропии получила развитие в работах Вершика, Петрова и Затицкого \cite{VPZ, Z1, PZ, Z2}. Напомним основные понятия и утверждения этой теории.

Let $\rho\colon (X^2, \mu^2) \to [0, +\infty)$ be a measurable semimetric on a measure space $\lr{X,\mu}$. That is, $\rho$~is a non--negative symmetric function which is measurable with respect to $\mu^2$ and satisfies the triangle inequality.  For a positive $\eps$, we define its \emph{$\eps$--entropy} as follows. Let $k$ be the minimal positive integer such that $X$ can be represented as a union of measurable subsets $X_0, X_1, \ldots, X_k$ with $\mu(X_0) < \eps$ and $\diam_\rho(X_i) < \eps$ for all $i>0$. Put
%Пусть $\rho\colon (X^2, \mu^2) \to [0, +\infty)$ --- измеримая полуметрика на пространстве с мерой $\lr{X,\mu}$, {то есть, измеримая по мере $\mu^2$ неотрицательная симметричная функция, удовлетворяющая неравенству треугольника}. Для положительного $\eps$ определим её \emph{$\eps$--энтропию} $\mbbH_\eps(X, \mu, \rho)$ следующим образом. Пусть $k$ --- наименьшее натуральное число, для которого пространство $X$ можно представить в виде объединения измеримых множеств $X_0, X_1, \ldots, X_k$, таких, что $\mu(X_0) < \eps$ и при $i = 1, \ldots, k,$ диаметр множества $X_i$ в полуметрике $\rho$ меньше $\eps$. Положим
\[
\mbbH_\eps(X, \mu, \rho) = \log_2 k.
\]
If there is no such finite $k$, define $\mbbH_\eps(X, \mu, \rho) = + \infty$.
%Если же такого $k$ не существует, положим $\mbbH_\eps(X, \mu, \rho) = + \infty$.    

We call a semimetric \emph{admissible} if it is separable on some subset  of full measure. Properties of admissible semimetrics are studied in detail in~\cite{VPZ}. In particular, it is proved that a semimetric is admissible if and only if all its $\eps$--entropies are finite for all $\eps > 0$. 
%Полуметрика называется \emph{допустимой}, если она сепарабельна на некотором подмножестве $X_0 \subset X$, таком, что $\mu\lr{X\setminus X_0} = 0$. В работе \cite{VPZ} изучаются свойства допустимых полуметрик. В частности, доказано, что полуметрика допустима тогда и только тогда, когда её $\eps$--энтропия конечна при любом $\eps > 0$.

Suppose that $G \acts{} \lr{X,\mu}$ is a measure--preserving action of a countable group $G$ on a Lebesgue space $\lr{X,\mu}$. For an element $g \in G$ denote a shifted semimetric $g^{-1}\rho$:  $g^{-1}\rho(x,y) = \rho(gx, gy)$, where $x,y \in X$. Evidently, if $\rho$ is admissible, then $g^{-1}\rho$ is admissible as well. 
%Пусть группа $G$ действует автоморфизмами на стандартном вероятностном пространстве $\lr{X,\mu}$. Символом $g^{-1}\rho$ мы будем обозначать сдвинутую полуметрику: $g^{-1}\rho(x,y) = \rho(gx, gy)$, $x,y \in X$.  Ясно, что полуметрика $g^{-1}\rho$ допустима тогда и только тогда, когда полуметрика $\rho$ является таковой.

Let us fix a sequence $\lambda = \{S_n\}_{n=1}^{\infty}$ of non--empty finite subsets of~$G$. Here and in what follows, we call it the \emph{equipment} of the group. A measurable semimetric $\rho$ is called \emph{generating} if all its shifts by elements of $\cup_n S_n$ together separate points up to a null set. This means that there exists a subset $X_0 \subset X$ of full measure such that for any pair of distinct points $x, y \in X_0$ there is an element $g \in \cup_n S_n$ with $g^{-1}\rho(x, y) > 0$. Note that any actual (measurable) metric is always generating. Next, define an averaged by~$S_n$ semimetric $G^n_{av}\rho$ in a natural way:  
%Зафиксируем некоторую последовательность $\lambda = \{S_n\}_{n=1}^{\infty}$ конечных подмножеств группы~$G$, будем называть ее \emph{оснащением} группы~$G$. Измеримая полуметрика $\rho$ называется \emph{порождающей} относительно $(G,\lambda)$, если ее сдвиги под действием элементов $\cup_n S_n$ разделяют точки mod 0, т.\,e. существует такое подмножество $X_0 \subset X$ полной меры, что для любых различных $x, y \in X_0$ найдется такой элемент $g \in \cup_n S_n$, что $g^{-1}\rho(x, y) > 0$. Отметим, что измеримая метрика всегда является порождающей. Символом $G^n_{av}\rho$ мы будем обозначать усреднение сдвигов полуметрики $\rho$ под действием элементов множества~$S_n$:
\[
G^n_{av} \rho (x,y) = \frac{1}{\abs{S_n}}\suml_{g\in S_n} \rho(gx, gy), \qquad x,y \in X.
\]
Sometimes we will emphasize the set over which the averaging is taken. In this case, we will write~$G^{S_n}_{av}\rho$ instead of~$G^n_{av}\rho$.

Consider then the following function 
%Рассмотрим следующую величину 
\[
\Phi(n, \eps) = \mbbH_\eps\lr{X, \mu, G^n_{av}\rho}.
\]
Actually, the function $\Phi(n,\eps)$ depends on $n$, $\eps$, and semimetric $\rho$. However, its asymptotic behaviour is supposed to be independent of~$\rho$ and~$\eps$ in some sense (see~\cite{V1, V3}). Let us recall a definition proposed in~\cite{PZ, Z2}. 
%Априори, функция $\Phi(n,\eps)$ зависит от $n$, $\eps$ и полуметрики $\rho$. Однако, предполагается, что асимптотическое поведение по $n$ этой функции в некотором смысле не зависит от полуметрики~$\rho$ и числа~$\eps$ (см.~\cite{V1, V3}). Напомним определение, предложенное в работах~\cite{PZ, Z2}.
\begin{defn}\label{def_scalingseq}
        Let $G \acts{} (X,\mu)$ be a measure--preserving action of a group $G$ equipped with $\lambda$ and $\rho$ be an admissible semimetric on $(X,\mu)$. We call a sequence $\{h_n\}_{n =1 }^\infty$ \emph{a scaling entropy sequence} of this action of the equipped group $G$ and the semimetric $\rho$ if for all sufficiently small $\eps>0$ the following asymptotic relation holds: 
        %Пусть $G \acts{\alpha} (X, \mu)$ --- действие группы $G$ с оснащением $\lambda$ автоморфизмами на стандартном вероятностном пространстве $(X,\mu)$. Пусть $\rho$ --- допустимая полуметрика на~$(X,\mu)$. Последовательность $\{h_n\}_{n =1 }^\infty$ называется \emph{масштабирующей энтропийной последовательностью} для полуметрики $\rho$ и действия группы $G$ c оснащением $\lambda$, если при достаточно малых $\eps > 0$ 
        \[
        \mbbH_\eps(X,\mu, G^n_{av}\rho) \asymp h_n.
        \]
\end{defn}
Here, for two functions $\phi(n)$ and $\psi(n)$ relation $\phi(n) \asymp \psi(n)$ means that there are two positive constants $c$ and $C$ such that  $c \phi(n) \le \psi(n) \le C\phi(n)$. Note that it makes sense to consider the whole class of equivalent scaling entropy sequences. Indeed, it is easy to see from Definition~\ref{def_scalingseq} that if $\{h_n\}$ is a scaling sequence of the action and $h_n^\prime\asymp h_n$, then the sequence $\{h_n^\prime\}$ is scaling as well. 
%Здесь и далее для двух последовательностей $\phi$ и $\psi$ соотношение $\phi(n) \asymp \psi(n)$ означает, что существуют такие положительные константы $c$ и $ C$, что  $c \phi(n) \le \psi(n) \le C\phi(n)$. Отметим, что, вообще говоря, имеет смысл говорить сразу о классе масштабирующих последовательностей. Действительно, как видно из определения~\ref{def_scalingseq}, если последовательность $\{h_n\}$ является масштабирующей и для некоторой другой последовательности $\{h_n'\}$ имеет место соотношение $h_n'\asymp h_n$, то $\{h_n'\}$ тоже является масштабирующей.

{In \cite{Z1, Z2}, P.~Zatitskiy, proving a conjecture by A.~Vershik, showed that if a sequence $\{h_n\}$ is a scaling entropy sequence for some summable admissible metric $\rho$, then it is a scaling sequence for any other such metric}. The summability of $\rho$ means that it has finite integral over $X^2$, i.\,e., $\rho \in L^1(X^2,\mu^2)$. In particular, any bounded measurable semimetric is summable. This independence holds for any equipment $\lambda$ but $\rho$ must be an actual metric. The case of a generating semimetric imposes some additional requirements on $\lambda$ (see~\cite{Z2} for details). 
%В работах \cite{Z1, Z2} доказано, что, если последовательность $h_n$ является масштабирующей для какой-то суммируемой допустимой метрики $\rho$, то она является масштабирующей и для любой другой суммируемой допустимой метрики. Суммируемость полуметрики $\rho$ означает, что её интеграл по множеству $X^2$ конечен, то есть $\rho \in L^1(X^2,\mu^2)$, в частности, любая ограниченная измеримая полуметрика заведомо является суммируемой. Отметим, что это верно для любого оснащения $\lambda$, но рассматриваются лишь настоящие метрики. Для порождающих полуметрик необходимо наложить некоторые условия на оснащение $\lambda$ (см. работу~\cite{Z2}).
\begin{defn}\label{def_equip}
        Equipment $\lambda = \{S_n\}$ of a countable group $G$ is called \emph{suitable} if for any $g\in \cup S_n$ and $\delta > 0$ there exists $k \in \mbbN$ such that for all $n\in \mbbN$ there are $g_1, \ldots, g_k \in G$ with
        %Оснащение $\lambda = \{S_n\}$ группы $G$ называется \emph{подходящим}, если для любого $g\in \cup S_n$ и любого $\delta > 0$ существует такое $k \in \mbbN$, что для любого $n\in \mbbN$ найдутся такие $g_1, \ldots, g_k \in G$, что выполнено неравенство
        \[
        \bigg| gS_n \setminus \bigcup\limits_{j=1}^k S_n g_j\bigg| \le \delta \abs{S_n}.
        \]
\end{defn}
Note that any left F\o lner sequence forms suitable equipment. For any action of a suitably equipped group, it is proved in~\cite{Z1, Z2} that if a sequence $\{h_n\}$ is a scaling sequence for some generating admissible summable semimetric, then it is a scaling entropy sequence for all such semimetrics. Therefore, the class of scaling entropy sequences does not depend on the choice of semimetric  and forms a measure--theoretic invariant of an action of an equipped group. 
%Отметим, что любая (левая) последовательность Фёльнера аменабельной группы является подходящим оснащением. Для действия группы $G$ с подходящим оснащением в работах \cite{Z1, Z2} доказано, что если последовательность $h_n$ является масштабирующей для какой-то суммируемой допустимой порождающей полуметрики $\rho$, то она является масштабирующей и для любой другой суммируемой допустимой порождающей полуметрики. Таким образом, класс масштабирующих последовательностей не зависит от выбора полуметрики и образует метрический инвариант действия группы с оснащением. 

We should note that a scaling entropy sequence may depend on the choice of equipment. It is shown in~\cite{PZ} that under certain conditions for the equipment if a scaling sequence exits, then one could choose a subadditive increasing function $f\colon \mathbb{N} \to \mathbb{N}$ with $h_n\asymp f({\abs{S_n}})$. However, it is now unknown whether such $f$ can be chosen independently of the equipment. Moreover, it is unclear if the stability (see Section~\ref{sec_unstable}) depends on the choice of equipment. 
%Также необходимо отметить, что масштабирующая последовательность действия группы, вообще говоря, априори может зависеть от выбора оснащения. В работе \cite{PZ} показано, что при некоторых условиях на оснащение $\lambda$ счетной группы $G$, если масштабирующая энтропийная последовательность существует, то существует такая возрастающая субаддитивная функция $f\colon \mathbb{N} \to \mathbb{N}$, что $h_n\asymp f({\abs{S_n}})$. Однако, на данный момент не известно, можно ли выбрать такую функцию $f$ одновременно для всех оснащений. Более того, не известно, зависит ли стабильность (см. пункт \ref{sec_unstable}) системы от выбора оснащения $\lambda$. 

It is proved in \cite{PZ} that in the case of one transformation (i.\,e., action of $\mbbZ$ with the standard equipment $S_n = \{-n, \ldots, n\}$) if the class of scaling entropy sequences is non--empty, then it contains an increasing subadditive function. In~\cite{Z2}, the explicit examples of automorphisms with a given increasing subadditive scaling entropy sequence are given. In addition, the actions of $\bigoplus \mathbb Z_2$ with a given scaling sequence of an intermediate growth were constructed in~\cite{Z2}. This construction can be easily generalized to the case of the group $\bigoplus_{k} \mathbb{Z}_{r_k}$ for an arbitrary family of positive integers~$\{r_k\}$.
%В работе \cite{PZ} показано, что для действия одного автоморфизма (т.\,е. группы $\mathbb{Z}$ с естественным оснащением отрезками) класс масштабирующих последовательностей, если не является пустым, содержит возрастающую субаддитивную последовательность, а в работе \cite{Z2} были построены примеры автоморфизмов с наперед заданными возрастающими субаддитивными масштабирующими последовательностями, тем самым были полностью описаны непустые классы масштабирующих последовательностей автоморфизмов. Кроме того, в работе \cite{Z2} были построены примеры действий группы $\bigoplus \mathbb Z_2$ с наперед заданными масштабирующими последовательностями промежуточного роста, и {эта конструкция может быть с легкостью модифицирована для случая действия групп $\bigoplus_{k} \mathbb{Z}_{r_k}$ для произвольной последовательности натуральных чисел~$\{r_k\}$}.

\subsection{Stable and unstable systems. Examples of almost complete growth} \label{sec_unstable}
%\subsection{Стабильные и нестабильные системы, примеры почти полного роста} \label{sec_unstable}

It was recently shown by the author~\cite{Vep} that a scaling entropy sequence may not exist even for one automorphism. We will call a system \emph{stable} if its class of scaling entropy sequences is not empty. That is, in~\cite{Vep}, the examples of \emph{unstable}  {ergodic} systems were constructed.
%Как было недавно выяснено автором данной работы, масштабирующая последовательность существует не всегда, даже для действия группы $\mbbZ$. Системы, для которых масштабирующая последовательность существует, мы будем называть \emph{стабильными}. Недавно автором работы были построены примеры \emph{нестабильных} систем, т.\,е. таких систем, для которых класс масштабирующих последовательностей пуст.

However, the notion of a scaling sequence can be generalized to unstable cases. Let us define a partial order $\preceq$ on the set of functions from $\mathbb{N}\times \mathbb{R}_+$ to $\mathbb{R}_+$ {(that decrease with respect to their second arguments)} as follows:
%Понятие масштабирующей последовательности, однако, может быть обобщено и на нестабильные случаи. На множестве функций из $\mathbb{N}\times \mathbb{R}_+$ в $\mathbb{R}_+$ зададим отношение частичного порядка $\preceq$ следующим образом:
\begin{equation}
    \Psi \preceq \Phi \Longleftrightarrow \forall \eps > 0\ \exists \delta > 0 \ \Psi(n, \eps) \lesssim \Phi(n, \del).
\end{equation}
For two sequences $\phi(n)$ and $\psi(n)$,  we write $\phi \lesssim \psi$ if there is a positive constant $C$ such that $\phi(n) \le C \psi(n)$ for all $n \in \mbbN$. We will call two functions $\Psi$ and $\Phi$ equivalent if $\Psi \preceq \Phi$ and $\Phi \preceq \Psi$ hold simultaneously. We denote the equivalence class of this relation containing  $\Phi$ by~$\wpart{\Phi}$.
%Здесь и далее соотношение $\phi(n) \lesssim \psi(n)$ для двух последовательностей $\phi$ и $\psi$ означает, что существует такая положительная константа $C$, что $\phi(n) \le C \psi(n)$ при всех $n$. Две функции $\Psi$ и $\Phi$ назовём эквивалентными, если одновременно $\Psi \preceq \Phi$ и $\Phi \preceq \Psi$. Класс эквивалентности функции $\Phi$ относительно такого отношения мы будем обозначать символом~$\wpart{\Phi}$. 

Let $G \curvearrowright (X, \mu)$ be a measure--preserving action of a suitably equipped group $(G,\lambda)$ and $\rho$ be an admissible generating summable semimetric on $(X, \mu)$. In~\cite{Z1, Z2}, it is proved (see Lemma~9 in~\cite{Z1} and similar statements in~\cite{Z2}) that the equivalence class of function $\Phi_\rho(n, \eps) = \mbbH_\eps\lr{X, \mu, G^n_{av}\rho}$ does not depend on a semimetric and forms an invariant  $\mathcal{H}\lr{X,\mu, G, \lambda}$ of the measure--preserving action:
%Пусть $G \curvearrowright (X, \mu)$ --- действие группы $G$ с подходящим оснащением $\lambda$. Пусть $\rho$ --- некоторая допустимая порождающая суммируемая полуметрика на $(X, \mu)$. В работах \cite{Z1, Z2} доказано (см. лемму 9 работы \cite{Z1} и аналогичные утверждения работы \cite{Z2}), что класс эквивалентности функции $\Phi_\rho(n, \eps) = \mbbH_\eps\lr{X, \mu, G^n_{av}\rho}$ не зависит от полуметрики  $\rho$ и образует метрический инвариант $\mathcal{H}\lr{X,\mu, G, \lambda}$ действия группы $G$:
\begin{equation}
    \mathcal{H}\lr{X, \mu, G, \lambda} = \Big[\Phi_\rho(n, \eps)\Big].
\end{equation}
Note that the system is stable if and only if $\mathcal{H}\lr{X, \mu, G, \lambda}$ contains a function $\Phi(n, \eps) = \phi(n)$ independent of $\eps$.
%Отметим, что система является стабильной тогда и только тогда, когда в классе $\mathcal{H}\lr{X, \mu, G, \lambda}$ можно найти функцию $\Phi(n, \eps) = \phi(n)$, не зависящую от $\eps$. 

Scaling entropy may also depend on the choice of equipment. However, if $\lambda = \{S_n\}$ and $\theta = \{W_n\}$ are such that $|S_n \triangle W_n| = o(|S_n|)$, then $\mathcal{H}\lr{X, \mu, G, \lambda}=\mathcal{H}\lr{X, \mu, G, \theta}$ for any measure--preserving system~$(X, \mu, G)$.
%{Масштабированная энтропия также априори зависит от выбора оснащения. Однако, отметим, что если два подходящих оснащения $\lambda = \{S_n\}$ и $\theta = \{W_n\}$ таковы, что $|S_n \triangle W_n| = o(|S_n|)$, то $\mathcal{H}\lr{X, \mu, G, \lambda}=\mathcal{H}\lr{X, \mu, G, \theta}$ для любого сохраняющего меру действия $(X, \mu, G)$.} 

Theorem~\ref{thm_kolm} below states that for any amenable group $G$ equipped with a F\o lner sequence $\lambda = \{F_n\}$, any $\Phi \in \mathcal{H}\lr{X, \mu, G, \lambda}$, and any positive $\eps$ the  following asymptotic relation holds:
%В теореме \ref{thm_kolm} показано, что для любой аменабельной группы $G$ c оснащением фёльнеровской последовательностью $\lambda = \{F_n\}$, любой  $\Phi \in \mathcal{H}\lr{X, \mu, G, \lambda}$ и для любого $\eps >0$ имеет место асимптотическое соотношение
\begin{equation}
    \label{intro_eq_maxgrowth}
\Phi(n, \eps) \lesssim \abs{F_n}.
\end{equation}
The equivalence in~\eqref{intro_eq_maxgrowth} holds if and only if the amenable measure entropy of $(X,\mu,G)$ is positive. In~\cite{Z2}, the examples of {ergodic} stable $\mbbZ$--actions with almost complete growth (see Definition~\ref{def_almost}) with respect to the standard equipment are given.  
%Эквивалентность в~\eqref{intro_eq_maxgrowth} достигается в том и только том случае, когда метрическая энтропия системы $(X,\mu,G)$ положительна. В работе \cite{Z2} построены примеры стабильных систем для действия групп $\mbbZ$ почти почти полного роста (см. определение \ref{def_almost}) для стандартного оснащения отрезками. 

The main theorem of this work (Theorem~\ref{thm_amen}) states that such {ergodic} actions exist for any countable non--periodic amenable group with arbitrary F\o lner equipment. The non--existence of a universal zero entropy system for such groups is proved in Theorem~\ref{thm_univ} by means of constructed actions of almost complete growth.  
%Основная теорема  данной работы (теорема \ref{thm_amen}) гарантирует существование таких действий для любой непериодической аменабльной группы относительно любого Фёльнеровского оснащения. Несуществование универсальной системы нулевой энтропии для действия непериодических аменабельных групп доказывается в теореме \ref{thm_univ} с помощью построенных действий с масштабирующими последовательностями почти полного роста.

\section{Scaling Entropy and Amenable Measure Entropy}
%\section{Метрическая энтропия и масштабирущая последовательность} 

In this section, we study the relation between the notions of scaling entropy and amenable measure entropy.
%В этом разделе мы изучаем связь масштабированной энтропии с метрической энтропией.

Let us state several technical lemmas that we use in the proof of Theorem~\ref{thm_kolm} {below}. Note that for any measurable partition $\xi$ of a measure space $(X, \mu)$ there is a naturally defined \emph{cut semimetric} $\rho_{\xi}(x,y)$ which equals zero if $x$ and $y$ both lie in the same cell  of $\xi$ and one otherwise. If the partition is finite (or countable) up to a null set then the corresponding cut semimetric is admissible. The following lemma (see \cite{Z1}) links the $\eps$--entropy of $\rho_\xi$ with the Shannon entropy of~$\xi$.    
%Сначала сформулируем несколько технических лемм, необходимых для доказательства теоремы \ref{thm_kolm}. Каждому измеримому разбиению $\xi$ пространства $(X, \mu)$ канонически соответствует \emph{разрезная полуметрика} $\rho_{\xi}(x,y)$, принимающая значение $0$, если $x$ и $y$ лежат в одном элементе $\xi$, и значение $1$ иначе. Следующая лемма, доказанная в работе \cite{Z1}, связывает $\eps$--энтропию полуметрики $\rho_\xi$ с энтропией $H(\xi)$ измеримого разбиения. 
\begin{lm}\label{lm_partitions}
The following relations between $\eps$--entropy and Shannon entropy hold. 
%Справедливы следующие соотношения между энтропией разбиений и $\eps$--энтропией полуметрик.  
    \begin{enumerate}
        \item
        For any measurable partition $\xi$ of a standard measure space $\lr{X,\mu}$ and any $\eps > 0$  
        %Для любого измеримого разбиения $\xi$ стандартного вероятностного пространства $\lr{X,\mu}$ и любого $\eps > 0$ выполнено неравенство
        \[
        \mbbH_\eps\lr{X, \mu, \rho_\xi} \le \frac{H(\xi)}{\eps},
        \]
        where $\rho_\xi$ is a semimetric corresponding to $\xi$.
        %где $\rho_\xi$ --- соответствующая разбиению разрезная полуметрика. 
        
        \item 
        Let $m,k \in \mbbN$ and $\{\xi_i\}_{i=1}^k$ be a family of finite measurable partitions of $\lr{X,\mu}$ such that each~$\xi_i$ consists of not more than~$m$ cells.  Let $\xi = \bigvee_{i=1}^k\xi_i$ be the refinement of these partitions and $\rho = \frac{1}{k} \sum_{i=1}^k\rho_{\xi_i}$ be the averaging of corresponding semimetrics. Then for any $\eps \in (0, \frac{1}{2})$ the following estimate holds:
        %Пусть $m,k \in \mbbN$ и пусть $\{\xi_i\}_{i=1}^k$ --- семейство конечных измеримых разбиений пространства $\lr{X,\mu}$, каждое из которых состоит из не более чем $m$ элементов. Пусть $\xi = \bigvee_{i=1}^k\xi_i$ --- произведение этих разбиений, $\rho = \frac{1}{k} \sum_{i=1}^k\rho_{\xi_i}$ --- усреднение  соответствующих этим разбиениям разрезных полуметрик. Тогда для любого $\eps \in (0, \frac{1}{2})$ справедлива следующая оценка
        \[
        \frac{H(\xi)}{k} \le \frac{\mbbH_\eps(X, \mu, \rho)}{k} + 2\eps \log m - \eps \log \eps - (1 - \eps)\log(1 - \eps) + \frac{1}{k}.
        \]
    \end{enumerate}
\end{lm}

The next lemma (see \cite[Lemma 1]{PZ}) gives an upper bound for the $\eps$--entropy of an averaged semimetric.
%Следующая техническая лемма доказана в работе \cite{PZ}, она даёт верхнюю оценку $\eps$--энтропии усреднений полуметрик. 
\begin{lm}\label{lm_upperbound}
    Let $\rho_1, \rho_2, \ldots, \rho_k$ be admissible semimetrics on $(X, \mu)$ with $\rho_i \le 1$ for all $i = 1, \ldots, k$. Assume that $\eps \in (0,1)$ satisfies $\mbbH_\eps(X, \mu, \rho_i) > 0$. Then the following inequality holds
    %Пусть $\rho_1, \rho_2, \ldots, \rho_k$ --- допустимые полуметрики на $(X, \mu)$, причём $\rho_i \le 1$ для всех $i = 1, \ldots, k$. И пусть $\eps \in (0,1)$ таково, что $\mbbH_\eps(X, \mu, \rho_i) > 0$. Тогда выполнено неравенство 
    \[
     	\mbbH_{2\sqrt{\eps}}\left(X, \mu, 
     	\frac{1}{k} \suml_{ i= 1}^{k}\rho_i \right) \le
     	2 \suml_{ i= 1}^{k} \mbbH_\eps(X, \mu, \rho_i). 
    \] 
\end{lm}

The theorem below relates notions of scaling entropy and classical measure entropy for amenable groups. The similar result for one automorphism was proved in~\cite{Z1}.   
%{Следующая теорема, связывающая масштабирующую последовательность с классической метрической энтропией, является аналогом теоремы 7 работы \cite{Z1} для случая действия аменабельной группы с оснащением фёльнеровской последовательностью.}

\begin{thm}\label{thm_kolm}
Let $G \acts{} \lr{X, \mu}$ be a measure--preserving action of an amenable group and $\lambda = \{F_n\}$ be a F\o lner sequence in $G$. Choose some $\Phi \in \mathcal{H}\lr{X, \mu, G, \lambda}$.
%Пусть аменабельная группа $G$ действует автоморфизмами на стандартном вероятностном пространстве $\lr{X, \mu}$. И пусть $\lambda = \{F_n\}$ --- последовательность Фёльнера группы $G$. Пусть $\Phi \in \mathcal{H}\lr{X, \mu, G, \lambda}$. 
\begin{enumerate}
    \item 
    Assume that the amenable measure entropy $h(X,\mu, G)$ is positive. Then~$\lr{X, \mu, G}$ is stable and for all sufficiently small $\eps > 0$ 
    %Предположим, что метрическая энтропия $h(X,\mu, G)$ положительна. Тогда система $\lr{X, \mu, G}$ стабильна, и {{при достаточно малых $\eps$}}
    \[
    \Phi\lr{n, \eps} \asymp \abs{F_n}.
    \]
    \item 
    If $h(X, \mu, G) = 0$, then for all positive $\eps$
    %Если же $h(X, \mu, G) = 0$, то при всех  $\eps$
    \[
    \Phi\lr{n, \eps} = o \lr{\abs{F_n}}.
    \]
\end{enumerate}
\end{thm}
\begin{proof}
    Suppose that the classical entropy $h(X, \mu, G)$ is positive. Let $\xi$ be a finite measurable partition and $\zeta_n = \bigvee_{g \in F_n} g^{-1}\xi$. Let  $\rho_\xi$ be the corresponding to $\xi$ cut semimetric. Also, for $g \in F_n$ set $\xi_g = g^{-1}\xi$ and  $m = \abs{\xi} = \abs{\xi_g}$. Due to part~2 of Lemma~\ref{lm_partitions}, we have: 
    %Предположим, что метрическая энтропия $h(X, \mu, G)$ положительна. Пусть $\xi$ есть некоторое конечное измеримое разбиение, $\zeta_n = \bigvee_{g \in F_n} g^{-1}\xi$. Пусть $\rho_\xi$ --- разрезная полуметрика, соответствующая $\xi$. Также, пусть $\xi_g = g^{-1}\xi$ для $g \in F_n$ и $m = \abs{\xi} = \abs{\xi_g}$ --- количество элементов разбиения. По 2 пункту леммы \ref{lm_partitions} имеем:
    \begin{equation}\label{thm_kolm_eq_positive}
    \frac{H\left(\zeta_n\right)}{\abs{F_n}} \le 
    \frac{\mbbH_\eps(X, \mu, G_{av}^n\rho_\xi)}{\abs{F_n}} + 
    2\eps \log m - \eps \log \eps - (1 - \eps)\log(1 - \eps) + \frac{1}{\abs{F_n}}.
    \end{equation}
    Since $h(X, \mu, G)> 0$, we can choose $\xi$ satisfying  
    %В силу того, что $h(X, \mu, G)> 0$, разбиение $\xi$ можно выбрать так, что
    \[
    h(\xi) = \liml_{n\to+\infty} \frac{1}{\abs{F_n}}H\left(\zeta_n\right) > 0.
    \]
    Hence, while $\eps$ is small enough, inequality~\eqref{thm_kolm_eq_positive} implies 
    %Значит, из неравенства~\eqref{thm_kolm_eq_positive} при достаточно малом $\eps$ следует, что
    \begin{equation}
        \mbbH_\eps(X, \mu, G_{av}^n\rho_\xi) \gtrsim \abs{F_n}.
    \end{equation}
    Note that $\rho_\xi$ may not be generating (it happens if and only if $\xi$ is not generating). We can easily overcome this by adding some admissible metric  $\omega$ to $\rho_\xi$. Evidently,  $\omega + \rho_\xi$ is now generating and $\Psi(n, \eps) = \mbbH_\eps(X, \mu, G_{av}^n(\rho_\xi + \omega))$ lies in $\mathcal{H}(X,\mu, G, \lambda)$. At the same time,
    %Отметим, что полуметрика $\rho_\xi$ может не быть порождающей (если разбиение $\xi$ не является порождающим), однако для любой измеримой допустимой метрики $\omega$ на $X$ функция $\rho_\xi + \omega$ также является измеримой допустимой метрикой. Тогда $\Psi(n, \eps) = \mbbH_\eps(X, \mu, G_{av}^n(\rho_\xi + \omega))$ лежит в  $\mathcal{H}(X,\mu, G, \lambda)$. При этом 
    \[
    \Psi(n, \eps) \ge \mbbH_\eps(X, \mu, G_{av}^n\rho_\xi) \gtrsim \abs{F_n}.
    \]
    
    The upper bound for the asymptotics of the scaling entropy instantly follows from Lemma~\ref{lm_upperbound}. Indeed, let $\rho \le 1$ be an admissible generating semimetric. Using Lemma~\ref{lm_upperbound} for semimetrics $g^{-1}\rho$, where $g \in F_n$, we obtain   
    %Верхняя оценка на рост масштабированной энтропии сразу следует из леммы \ref{lm_upperbound}. Действительно, пусть $\rho \le 1$ --- некоторая допустимая порождающая полуметрика. Применяя лемму \ref{lm_upperbound} для полуметрик $g^{-1}\rho$ при $g \in F_n$ получаем
    \[
    \mbbH_{2\sqrt{\eps}}\lr{X, \mu, G_{av}^n \rho} \le 2\abs{F_n} {\mbbH_\eps(X, \mu, \rho)}.
    \]
    Therefore, for any $\Phi \in \mathcal{H}(X,\mu, G, \lambda)$ and any sufficiently small (and thus for all as well) $\eps > 0$,
    %Это означает, что для любой $\Phi \in \mathcal{H}(X,\mu, G, \lambda)$ и достаточно малого (а тогда и для любого) $\eps > 0$
    \[
    \Phi(n, \eps) \lesssim \abs{F_n}.
    \]
    The first part is proved.
    %{Тем самым, первый пункт теоремы доказан.}
    
    Now assume that $h(X, \mu, G) = 0$. Consider a generating partition $\xi$ with finite entropy and the corresponding (generating) semimetric $\rho_\xi$. Let  $\zeta_n = \bigvee_{g \in F_n} g^{-1}\xi$ be as above. The first part of Lemma~\ref{lm_partitions} implies the following inequality 
    %Пусть теперь энтропия $h(X, \mu, G)$ равна нулю. В этом случае рассмотрим порождающее разбиение $\xi$ конечной энтропии и соответствующую ему (порождающую) допустимую полуметрику $\rho_\xi$. Как и раньше, пусть $\zeta_n = \bigvee_{g \in F_n} g^{-1}\xi$. Первый пункт леммы \ref{lm_partitions} гарантирует следующее неравенство 
    \[
    \mbbH_\eps\lr{X, \mu, G_{av}^n\rho_\xi} \le \mbbH_\eps\lr{X, \mu, \rho_{\zeta_n}} \le \frac{H(\zeta_n)}{\eps}. 
    \]
    However, $h(X, \mu, G) = 0$ means that $H(\zeta_n) = o\lr{\abs{F_n}}$. Therefore, for any $\Phi \in \mathcal{H}(X,\mu, G, \lambda)$ and any $\eps > 0$
    %Однако, равенство $h(X, \mu, G) = 0$ означает, что $H(\zeta_n) = o\lr{\abs{F_n}}$, а значит,  для любой $\Phi \in \mathcal{H}(X,\mu, G, \lambda)$ и любого $\eps > 0$
    \[
    \Phi(n, \eps) = o\lr{\abs{F_n}}.
    \]
    The second part is proved.
\end{proof}

\section{Proof of the non--existence of a universal system}
%\section{Несуществование универсальной системы нулевой энтропии}

In this section, we prove the non--existence of a universal zero entropy system for actions of a non--periodic amenable group. However, the proof that we provide deals with a wider class which, we believe, coincides with the class of all amenable groups.\footnote{At least, it contains $\bigoplus \mathbb Z_2$ which is a torsion group.} 
%В этой секции мы доказываем  несуществование универсальной системы нулевой энтропии для действий непериодических аменабельных групп. Приведенное ниже доказательство, однако, имеет дело с более общим классом аменабельных групп. 

\begin{defn}\label{def_almost}
        We say that a group $G$ equipped with $\lambda = \{F_n\}$ \emph{admits actions of almost complete growth} if for any non--negative function $\phi(n) = o\lr{\abs{F_n}}$ there exists a measure--preserving system~$(X, \mu, G)$ such that for any $\Phi \in \mathcal{H}(X, \mu, G, \lambda)$ and for any sufficiently small $\eps$ the following holds: 
        %Мы будем говорить, что группа  $G$  с оснащением  $\lambda = \{F_n\}$ \emph{допускает действия почти полного роста}, если для любой неотрицательной функции $\phi(n) = o\lr{\abs{F_n}}$ существует такая система $(X, \mu, G)$, что для любой $\Phi \in \mathcal{H}(X, \mu, G, \lambda)$ и достаточно малого $\eps$ выполнено 
        \[
         \Phi(n,\eps) \not \lesssim \phi(n) \text{ and } \Phi(n,\eps) =o(|{F_n}|).
        \]
\end{defn}
By virtue of Theorem~\ref{thm_kolm}, the second condition in Definition~\ref{def_almost}  is equivalent to $h(X, \mu, G) = 0$ for amenable groups.
%Отметим, что в силу теоремы \ref{thm_kolm} второе условие в определении \ref{def_almost} эквивалентно тому, что метрическая энтропия системы $(X, \mu, G)$ равна нулю.

\begin{thm}\label{thm_univ1}
    Let $G \acts{} (X,d)$ be a continuous action of an amenable group $G$ on a compact metric space. Suppose that $G$ admits {ergodic} actions of almost complete growth for some F\o lner equipment $\theta = \{W_n\}$. Assume that for any {ergodic} measure--preserving system $(Y, \nu, G)$ with zero entropy there exists an invariant measure $\mu$ on $X$ with  
    %Пусть аменабельная группа $G$ действует гомеоморфизмами на метрическом компакте $(X,d)$. Предположим, что $G$ допускает действия почти полного роста для некоторого фёльнеровского оснащения $\theta = \{W_n\}$. Предположим, что для любой динамической системы $(Y, \nu, G)$ нулевой метрической энтропии существует мера $\mu$ на $X$, инвариантная относительно $G$, такая что
    \[
    (X, \mu, G) \cong (Y, \nu, G).
    \]
    Then the topological entropy of $(X,d,G)$ is positive.  
    %Тогда топологическая энтропия системы $(X,d,G)$ положительна.
\end{thm}

\begin{proof}%[Доказательство теоремы~\ref{thm_univ1}]
    We will prove the theorem by contradiction. Assume that $h_{top}(X, G) = 0$. Then
    %Будем доказывать от противного. Предположим, что $h_{top}(X, G) = 0$. Тогда 
    \[
        \supl_{\eps > 0} \limsup\limits_{n\to +\infty} \frac{1}{\abs{W_n}}\log \spn(d, W_n, \eps) = 0. 
    \]
    Hence, for all $\eps > 0$
    %Стало быть, для всех $\eps > 0$
    \[
    \limsup\limits_{n\to +\infty} \frac{1}{\abs{W_n}}\log \spn(d, W_n, \eps) = 0.
    \]
    It is clear that there exists a function $\phi(n)$ such that $\frac{\phi(n)}{\abs{W_n}} \to 0$ and for any $\eps > 0$
    %Ясно, что существует такая функция $\phi(n)$, что $\frac{\phi(n)}{\abs{W_n}} \to 0$  и для всех $\eps > 0$
    \[
    \phi(n) \gtrsim \log \spn(d, W_n, \eps).
    \]
    By the theorem assumption, $G$ admits {ergodic} actions of almost complete growth with respect to~$\theta$. Therefore, there exists {an ergodic} measure--preserving action $G \acts{\alpha} (Y,\nu)$ with zero entropy such that for any $\Phi \in \mathcal{H}(Y, \nu, G, \theta)$ and any sufficiently small $\eps$
    %По предположению, группа $G$ с оснащением $\theta$ допускает действия почти полного роста. Значит,  существует такое действие  $\alpha\colon G\curvearrowright (Y,\nu)$ нулевой энтропии, что для любой $\Phi \in \mathcal{H}(Y, \nu, G, \theta)$ и достаточно малого~$\eps$
    \begin{equation}\label{thm_univ1_eq_bound}
        \Phi(n,\eps) \not \lesssim \phi(n).
    \end{equation}
    Also by the theorem assumption, this action has a representation in the topological system~$(X, d, G)$. Hence, there exists an invariant measure~$\mu$ on~$X$ with $(X,\mu, G) \cong  (Y, \nu, G)$. In particular, scaling entropy classes $\mathcal{H}(Y, \nu, G, \theta)$ and  $\mathcal{H}(X, \mu, G, \theta)$ coincide. Note that the metric~$d$ on~$X$ is obviously admissible and summable. Therefore, 
    %Также по предположению теоремы, данное действие может быть реализовано в топологической системе $(X, d, G)$. То есть, существует такая инвариантная мера $\mu$ на $X$, что $(X,\mu, G) \cong  (Y, \nu, G)$. В частности, $\mathcal{H}(Y, \nu, G, \theta) = \mathcal{H}(X, \mu, G, \theta)$. Отметим, что метрика $d$ на~$X$, очевидно, допустима. Значит,
    \[
    \mbbH_\eps\lr{X, \mu, G_{av}^n d} \in \mathcal{H}(X, \mu, G, \theta).
    \]
    However,
    %Однако,
    \[
    \mbbH_{\eps}(X,\mu, G_{av}^nd) \le \mbbH_{\eps}(X,\mu, G_{max}^n d) \le \log \spn\lr{d, W_n, \frac{\eps}{2}} \lesssim \phi(n),
    \]
    and we have obtained a contradiction to~\eqref{thm_univ1_eq_bound}.
    %что противоречит соотношению \eqref{thm_univ1_eq_bound}.
\end{proof}

In view of Theorem~\ref{thm_univ1}, %for the non--existence of a universal continuous action of zero entropy for non--periodic amenable groups (i.\,e., Theorem~\ref{thm_univ}) 
it only remains to prove that non--periodic amenable groups admit {ergodic} actions of almost complete growth to achieve our goal, i.\,e., Theorem~\ref{thm_univ}. The rest of the work is devoted to proving this.    
%В силу теоремы~\ref{thm_univ1} для доказательства отсутствия универсального топологического действия нулевой энтропии непериодической аменабельной группы $G$ (то есть теоремы~\ref{thm_univ}) нам остается показать, что такая группа допускает действия почти полного роста. Этому посвящена оставшаяся часть работы.

\section{Coinduced actions and scaling entropy}
%\section{Коиндуцированные действия}

In order to construct the actions of almost complete growth, we implement the procedure of coinduction. 
%Для построения действий почти полного роста мы будем использовать конструкцию коиндуцирования.

\subsection{Coinduced actions}
%\subsection{Коиндуцированные действия и лемма о расслоении}

Let us recall the construction of a coinduced action. Let $G$ be a countable amenable group and  $H \le G$ be a subgroup of $G$. Let $H \acts{\alpha} (X, \mu)$ be a measure--preserving action of $H$. Consider the decomposition of $G$ into a disjoint union of left cosets of~$H$: 
%В данном пункте мы напомним конструкцию коиндуцированного действия. Пусть $G$ --- счетная аменабельная группа и пусть $H \le G$. Пусть $H \acts{\alpha} (X, \mu)$ --- действие группы $H$ автоморфизмами на стандартном вероятностном пространстве $(X, {\mu})$. Рассмотрим разложение $G$ на смежные классы по подгруппе $H$:
\[
G = \bigsqcup_{i = 0}^{\infty} g_i H,
\]
where $g_i$ are some representatives of the cosets. For further convenience, we choose $g_0 = e$. Consider a measure space  
%где $g_i$ --- представители смежных классов. Можно считать, что $g_0 = e$. Рассмотрим пространство 
\[
{(X^{\faktor{G}{H}},\mu^{\faktor{G}{H}}) = \prod_{i=0}^{\infty} (X_{i},\mu_i)},
\]
where each ${(X_{i},\mu_i)}$ is an isomorphic copy of ${(X,\mu)}$ corresponding to $g_i$. For $x \in X^{\faktor{G}{H}}$ and $i\geq 0$, we denote the $i$-th coordinate of $x$ by $x_i$, $x_i \in X_i$. For any $g \in G$ and any $i\geq 0$, there are unique $k(i,g) \in \mbbN \cup \{0\}$ and $h(i,g)\in H$ with $g g_i = g_{k(i,g)} h(i,g)$. Define a coinduced action $\Ind_H^G\alpha$ of the group $G$ on a measure space $X^{\faktor{G}{H}}$ in the following way. Let $x \in X^{\faktor{G}{H}}$, put
%где ${(X_{i},\mu_i)}$ --- изоморфная копия ${(X,\mu)}$, соответствующая элементу $g_i$. Для $x \in X^{\faktor{G}{H}}$ символом $x_i$, $i\geq 0$, будем обозначать его $i-$ю координату, $x_i \in X_i$. Для любого $g \in G$ и каждого $i$ существуют единственные $k(i,g) \in \mbbN \cup \{0\}$ и $h(i,g)\in H$, такие, что $g g_i = g_{k(i,g)} h(i,g)$. Определим действие $\Ind_H^G\alpha$ группы $G$ на пространстве $X^{\faktor{G}{H}}$ следующим образом. Пусть $x \in X^{\faktor{G}{H}}$, положим
\[
g(x)_{i} = h(i, g^{-1})^{-1} \lr{x_{{k\lr{i, g^{-1}}}}}.
\]

Once and for all, we fix a system of cosets representatives $\{g_i\}$. 
%The following lemma shows that almost all sections of a F\o lner sequence uniformly satisfy the F\o lner condition.  
%Здесь и далее мы отождествляем множество классов смежности с некоторой фиксированной трансверсалью $\{g_i\}$.
%Определим фактор отображение $pr\colon G \to \faktor{G}{H}$, сопоставляющее каждому элементу $g\in G$ его класс смежности $gH$. 

\begin{lm}\label{lm_reduction}
    Let $H$ be a subgroup of a countable amenable group $G$  and $\{\tilde W_n\}_{n=1}^\infty$ be a F\o lner sequence in $G$. Then there exists another F\o lner sequence $\{W_n\}_{n=1}^\infty$ in  $G$ such that $|{W_n \triangle \tilde W_n}| = o(|{W_n}|)$ and
    %Пусть $H$ -- подгруппа счетной аменабельной группы $G$ и $\tilde W_n$ --- последовательность Фёльнера в $G$. Тогда существует последовательность Фёльнера $\{W_n\}_{n=1}^\infty$ в группе $G$, такая, что $\abs{\tilde W_n \triangle W_n} = o(\abs{W_n})$, и $W_n$ представима в виде
    \[
    W_n = \bigcup S_n^i g_i^{-1},
    \]
    where $S_n^i \subset H$ satisfy the following condition. For any  $h \in H$, $\eps > 0$, the inequality
    %где $S_n^i \subset H$ таковы, что для любого $h \in H$, $\eps > 0$ и достаточно больших $n$
    \[
    \abs{h S_n^i \triangle S_n^i} \le \eps {\abs{S_n^i}}
    \]
    holds for all sufficiently large $n$ and for any $i$.
\end{lm}

\begin{proof}
    Consider some $h \in H$ and $n \in \mbbN$. Let us denote by $\eps(n, h)$ the following value  
    %Рассмотрим некоторое $h \in H$ и $n \in \mbbN$. Символом $\eps(n, h)$ обозначим величину 
    $$
    \eps(n, h) = \frac{|h \tilde W_n \triangle \tilde W_n|} {|\tilde W_n|}.
    $$
    Due to the F\o lner condition,  $\eps(n, h)$ goes to zero when $h$ is fixed. Moreover, each $\tilde W_n$ uniquely decomposes into a disjoint union $\bigcup_i S_n^i g_i^{-1}$, where $S_n^i$ are some finite subsets of $H$.  Let  $\tilde I(n, h)$ be a set of integers~$i$ such that $\abs{h S_n^i \triangle S_n^i} > \eps^{\frac{1}{2}}{(n,h)} {\abs{S_n^i}}$. Denote $E(n, h) = \cup_{i\in \tilde I(n, h)} S_n^i g_i^{-1}$. The left multiplication by $h$ preserves all right cosets~$Hg_i^{-1}$. Therefore, 
    %Ясно, что $\eps(n, h)$ стремится к $0$ при любом фиксированном $h$. Также, множество $\tilde W_n$ представляется единственным образом в виде $\bigcup_i S_n^i g_i^{-1}$, где $S_n^i$ есть некоторые конечные подмножества $H$. Пусть $\tilde I(n, h)$ есть множество тех индексов $i$, для  которых $\abs{h S_n^i \triangle S_n^i} > \eps^{\frac{1}{2}}{(n,h)} {\abs{S_n^i}}$, и пусть $E(n, h) = \cup_{i\in \tilde I(n, h)} S_n^i g_i^{-1}$. Левое умножение на $h$ сохраняет каждый правый смежный класс~$Hg_i^{-1}$. Следовательно,
    \[
    \abs{E(n, h)} = \suml_{i\in \tilde I(n, h)}  \abs{S_n^i} < \eps{(n,h)}^{\frac{1}{2}}\abs{\tilde W_n}.
    \]
    Let $\tau\colon H \to \mbbN$ be an arbitrary enumeration of all elements of $H$. Define
    %Пусть $\tau\colon H \to \mbbN$ --- произвольная нумерация элементов группы $H$. Определим
    \[
    W_n = \tilde W_n \setminus \bigcup_{h \colon \eps(n,h) < 2^{-\tau(h)}} E(n,h).
    \]
   Clearly, sequence  $\{W_n\}$ is the desired one. Indeed, we have
   %Ясно, что последовательность $\{W_n\}$ является искомой. Действительно, имеем
   \[
   \abs{\tilde W_n \triangle W_n}  \le 
   \suml_{\eps(n,h) < 2^{-\tau(h)}} \abs{E(n, h)} <
   \suml_{\eps(n,h) < 2^{-\tau(h)}} \eps(n,h)^{\frac{1}{2}}\abs{\tilde W_n} = o(|\tilde W_n|).
   \]
   The last equality holds due to the Lebesgue's dominated convergence theorem.
   %Последнее соотношение справедливо, например, в силу теоремы Лебега. 
\end{proof}

\begin{defn}
	A subset $S$ of integer numbers is said to be \emph{$\eps$--invariant} for some positive $\eps$ if it satisfies $\abs{(S+1)\triangle S} < \eps \abs{S}$.
	%Множество $S$ целых чисел назовём $\eps$--инвариантным для некоторого положительного $\eps$, если $\abs{(S+1)\triangle S} < \eps \abs{S}$.
\end{defn}

\begin{rem}
    Evidently, if $H = \mbbZ$ in Lemma~\ref{lm_reduction}, then there exists a subsequence $\{n_j\}$ such that all $S_{n_j}^i$ are $\frac{1}{j}$--invariant. 
    %Ясно, что существует такая подпоследовательность $n_j$, что все множества $S_{n_j}^i$ являются $\frac{1}{j}$--инвариантными. 
\end{rem}

\subsection{Scaling entropy of a coinduced action}
%\subsection{Масштабированная энтропия коиндуцированного действия}

In this section, we estimate the scaling entropy of a conduced action. We reduce the question about the existence of almost complete actions for non--periodic amenable groups to the case of the group $\mathbb{Z}$ which is considered in Section~\ref{S6}. Also, we use an important technical Lemma~\ref{lm_estimate} whose proof is postponed to Section~\ref{sec_proof_est}.  
%В этом разделе мы {приводим оценки масштабированной энтропии коиндуцированного действия. Мы сводим вопрос существования действий почти полного роста для произвольной непериодической аменабельной группы к случаю группы $\mathbb{Z}$, который будет рассмотрен в разделе~\ref{S6}.}

\begin{thm}\label{thm_amen}
Let $\lambda = \{F_n\}_{n=1}^\infty$ be a F\o lner sequence of a countable non--periodic amenable group~$G$. Then the group~$G$ admits {ergodic} actions of almost complete growth with respect to~$\lambda$. 
%Пусть $\lambda = \{F_n\}_{n=1}^\infty$ --- произвольная последовательность Фёльнера счётной непериодической аменабельной группы $G$. Тогда группа $G$ с оснащением $\lambda$ допускает действия почти полного роста.  
\end{thm}

%\begin{rem}
%	The actions of almost complete growth that we provide below are ergodic.  
%	%Построенные ниже действия группы $G$ почти полного роста являются эргодическими.
%\end{rem}

%\begin{proof}[Proof of Theorem~\ref{thm_amen}]%[Доказательство теоремы~\ref{thm_amen}]
\begin{proof}
    Let $h \in G$ be an element of infinite order in $G$ and  $H = \left< h\right>$ be the subgroup generated by~$h$. Let $\{g_i\}$ be a system of representatives of left cosets with $g_0 = e$. Lemma~\ref{lm_reduction} states that there is a sequence $\theta = \{W_n\}$ of finite subsets of $G$ with
    %Пусть $h \in G$ элемент бесконечного порядка и $H = \left< h\right>$ --- порождённая им подгруппа, а $\{g_i\}$ --- трансверсаль к ней, причём $g_0 = e$. Лемма \ref{lm_reduction} утверждает, что существует такая последовательность Фёльнера $\theta = \{W_n\}$ в группе $G$, что  
    \begin{equation}\label{eq140801}
        \abs{F_n \triangle W_n} = o (\abs{F_n}),
    \end{equation}
    and $W_n = \bigcup S_n^i g_i^{-1}$, where $S_n^i \subset H$ are such that for any $\eps > 0$ and for any sufficiently large~$n$ 
    %и  $W_n = \bigcup S_n^i g_i^{-1}$, где $S_n^i$ таковы, что для любого $\eps > 0$ и достаточно больших $n$
    \[
    \abs{h S_n^i \triangle S_n^i} \le \eps {\abs{S_n^i}}.
    \]
    Relation~\eqref{eq140801} implies that any measure--preserving action $\alpha$ of $G$ satisfies $\mathcal{H}(\alpha, \lambda) = \mathcal{H}(\alpha, \theta)$. Thus, it suffices to prove that the group $G$ equipped with $\theta$ instead of $\lambda$ admits {ergodic} actions of almost complete growth. 
    %Соотношение~\eqref{eq140801} гарантирует, что для любого сохраняющего меру действия $\alpha$ группы $G$ выполняется равенство $\mathcal{H}(\alpha, \lambda) = \mathcal{H}(\alpha, \theta)$. Таким образом, достаточно доказать, что группа $G$ с оснащением $\theta$ допускает действия почти полного роста. 
    
     Let $\phi(n)$ be an increasing positive function which goes to infinity. We will apply the following lemma proved in Section~\ref{S6}.
     %Пусть $\phi(n)$ --- некоторая неубывающая функция, стремящаяся к бесконечности. Воспользуемся следующей леммой, доказанной в пункте \ref{S6}.
    
    \begin{lm}\label{lm_example}
	    Suppose that a sequence $\{S_n^i\}_{i = 1}^{k_n}$ of finite families of finite subsets of $\mbbZ$ is such that every $S_n^i$ is $\frac{1}{n}$--invariant. Let $\phi(n)$ be a sequence of positive numbers with $\lim_{n \to \infty} \phi(n) = \infty$. Then there exist an ergodic automorphism $T$ of a Lebesgue space $(X,\mu)$ and a subsequence  $\{n_j\}$ such that for any generating admissible summable semimetric  $\rho$ and sufficiently small $\eps > 0$ the following relation holds:
	    %Пусть дана последовательность конечных семейств $\{S_n^i\}_{i = 1}^{k_n}$ конечных подмножеств $\mbbZ$, такая, что каждое множество $S_n^i$ является $\frac{1}{n}$--инвариантным. Пусть также $\phi(n)$ --- некоторая последовательность положительных чисел, возрастающая к бесконечности. Тогда существует автоморфизм $T$ стандартного вероятностного пространства $(X,\mu)$ и последовательность $\{n_j\}$, такие, что для любой порождающей полуметрики $\rho$, достаточно малого $\eps > 0$ справедливо следующее соотношение:
    	\begin{equation}\label{lm_exampe_st}
    	\frac{\abs{S_{n_j}^i}}{\phi(n_j)} \lnsim \mbbH_\eps\lr{X, \mu, T_{av}^{S_{n_j}^i} \rho} \lnsim \abs{S_{n_j}^i},\qquad i = 1, \ldots, k_{n_j}.
    	\end{equation}     
    \end{lm}
    
    The symbol $\lnsim$ here means that the ratio of the left hand side to the right hand side tends to zero when~$j$ goes to infinity and this convergence is uniform with respect to~$i$. Note that there are only finitely many $S_{n_j}^i$ for each $j$. {We apply Lemma~\ref{lm_example} for the sequence of non-empty sets $S_n^i$ given by Lemma~\ref{lm_reduction}. Let~$I_n$ be the set of those indices $i$ for which $S_n^i$ is not empty and $k_n = |I_n|$. } Relation~\eqref{lm_exampe_st} implies, in particular, that  $T$ has zero entropy and for any $\eps$ small enough for sufficiently large~$j$ 
    %В частности, автоморфизм $T$ имеет нулевую энтропию, и для любого $\eps$ при достаточно большом $j$
    \begin{equation}\label{thm_amen_eq_subgrowth}
    \frac{\abs{S_{n_j}^i}}{\phi(n_j)} < \mbbH_{4\eps}\lr{X, \mu, T_{av}^{S_{n_j}^i} \rho} < \abs{S_{n_j}^i},\ 
    %i = 1, \ldots, k_{n_j}
    {i \in I_{n_j}}.
    \end{equation} 
    Consider  an action $\alpha = \Ind_H^G T$ coinduced from $H$ to the whole group $G$. Let $\tilde{\rho} \le 1$ be an admissible metric on $(X, \mu)$. Define a semimetric $\rho$ on $(X,\mu)^{\faktor{G}{H}}$ in the following way:
    %Рассмотрим действие $\alpha = \Ind_H^G T$ группы $G$, коиндуцированное с подгруппы $H$. Пусть $\tilde{\rho} \le 1$ --- некоторая допустимая метрика на $(X, \mu)$. Определим допустимую порождающую полуметрику $\rho$ на $(X,\mu)^{\faktor{G}{H}}$ следующим образом:
    \[
    \rho(x, y) = \tilde{\rho}(x_0, y_0), \qquad x,y \in X^{\faktor{G}{H}}.
    \]
    Although $G$ acts transitively on $\faktor{G}{H}$, the semimetric $\rho$ may not be generating.\footnote{For example, if all $S_n^0$ are empty.} However, it does not really matter because we are only looking for lower bounds for the scaling entropy. Since by the choice of the representatives $g_0 = e$, elements of $H$ act on the first component  of $(X,\mu)^{\faktor{G}{H}}$ independently of other coordinates. Thus, for all $x,y \in X^{\faktor{G}{H}}$ 
    %Так как $G$ действует на $\faktor{G}{H}$ транзитивно, то $\rho$ является порождающей.  Отметим, что представители классов смежности были выбраны так, что $g_0 = e$, следовательно элементы $H$ не переставляют нулевую координату $(X,\mu)^{\faktor{G}{H}}$. Таким образом, для любых $x,y \in X^{\faktor{G}{H}}$
    \[
    \rho (hx, hy)= \tilde{\rho}(hx_{0}, hy_{0}). 
    \]
    Then
    %Тогда 
    \[
    H_{av}^{S^i_{n_j}}\rho(x,y) = H_{av}^{S^i_{n_j}}\tilde{\rho}(x_{0},y_{0}).
    \]
    For each coset representative $g_i$, define a semimetric $\rho_i$ on $(X,\mu)^{\faktor{G}{H}}$ by 
    %Для каждого $g_i$ (выделенного представителя класса смежности) определим полуметрику $\rho_i$ на $(X,\mu)^{\faktor{G}{H}}$:
    \[
    \rho_i = g_i H_{av}^{S^i_{n_j}}\rho.
    \]
    Each $\rho_i$ depends only on the $i$-th component:  
    %Каждая полуметрика $\rho_i$  зависит только от $i$-ой координаты: 
    \[
    \rho_i(x,y) = (H_{av}^{S^i_{n_j}}\rho)(g_i^{-1}x,g_i^{-1} y) = (H_{av}^{S^i_{n_j}}\tilde{\rho})(x_i, y_i), \qquad x,y \in (X,\mu)^{\faktor{G}{H}}.
    \]
    Hence, we can consider $\rho_i$ as a semimetric on $X_i$. The averaging of  $\rho$ with respect to $W_{n_j}$ can be expressed in terms of $\rho_i$ as follows:
    %поэтому можно считать ее полуметрикой на $X_i$.  Усреднение полуметрики $\rho$ под действием $W_{n_j}$ выражается через $\rho_i$ следующим образом: 
    {
    \begin{multline}\label{eq98342}
     	%G_{av}^{W_{n_j}}\rho =  \frac{1}{\abs{W_{n_j}}}  \suml_{g_i} \suml_{s \in S^i_{n_j}} g_i s^{-1}\rho= 
    	%\frac{1}{\abs{W_{n_j}}}\suml_{g_i} g_i \suml_{s \in S^i_{n_j}} s^{-1} \rho = \\
    	%\frac{1}{\abs{W_{n_j}}} \suml_{g_i} \abs{S^i_{n_j}} g_i H_{av}^{S^i_{n_j}}{\rho} = 
    	%\frac{1}{\suml_{i} \abs{S^i_{n_j}}} \suml_{g_i} \abs{S^i_{n_j}} \rho_i.	
    	 G_{av}^{W_{n_j}}\rho =  \frac{1}{\abs{W_{n_j}}}  \suml_{i \in I_{n_j}} \suml_{s \in S^i_{n_j}} g_i s^{-1}\rho= 
    	\frac{1}{\abs{W_{n_j}}}\suml_{i \in I_{n_j}} g_i \suml_{s \in S^i_{n_j}} s^{-1} \rho = \\
    	\frac{1}{\abs{W_{n_j}}} \suml_{i \in I_{n_j}} \abs{S^i_{n_j}} g_i H_{av}^{S^i_{n_j}}{\rho} = 
    	\frac{1}{\suml_{i \in I_{n_j}} \abs{S^i_{n_j}}} \suml_{i \in I_{n_j}} \abs{S^i_{n_j}} \rho_i.
    \end{multline}}
  	
  	The next step is to estimate the epsilon--entropy of the semimetric given by~\eqref{eq98342}. The proof of the following lemma will be given in Section~\ref{sec_proof_est}. 
  	%Далее нам понадобится следующая лемма, доказанная в пункте \ref{sec_proof_est} этой работы.
  	\begin{lm}\label{lm_estimate}
		Suppose  $\eps > 0$ and $\phi > 1$ are fixed. 
		%Пусть $\eps > 0$ и $\phi > 1$ фиксированы. 
		%{(\color{red} При беглом просмотре кажется, что $\phi$ всегда стоит в знаменателе. Почему бы тогда не использовать $1/\phi$ вместо $\phi$?)}. 
		Consider a finite family of admissible semimetric triples $(X_i, \mu_i, \rho_i)$, $i = 1, \ldots, k$. Assume that  $\{s_i\}_{i=1}^k$ are such that $\phi^{-1} s_i < \mbbH_{4\eps}(X_i, \mu_i, \rho_i) < s_i$. Define a semimetric $\rho$ on $\prod_{i=1}^k (X_i, \mu_i) = (X, \mu)$ as a weighted averaging:
		%Рассмотрим конечное семейство допустимых полуметрческих троек $(X_i, \mu_i, \rho_i)$, $i = 1, \ldots, k$. Пусть $\{s_i\}_{i=1}^k$ таковы, что $\phi^{-1} s_i < \mbbH_{4\eps}(X_i, \mu_i, \rho_i) < s_i$. Зададим полуметрику $\rho$ на $\prod_{i=1}^k (X_i, \mu_i) = (X, \mu)$ следующим образом:
    		\[
    		\rho(x,y) = \frac{1}{\suml_{i=1}^k s_i} \suml_{i=1}^{k} s_i \rho_i(x_i, y_i),
    		\]
    	where $x = (x_1,\dots, x_k)$, $y = (y_1,\dots, y_k)$. Then    
	    %где $x = (x_1,\dots, x_k)$, $y = (y_1,\dots, y_k)$. Тогда 
    		\[
    		\mbbH_{\eps^4}(X,\mu, \rho) \ge \frac{1}{\phi}\eps^3  \suml_{i=1}^k \mbbH_{4\eps}(X_i, \mu_i, \rho_i) - k - 1. 
    		\]  
	\end{lm}
  	
  	Since for large~$j$ inequalities~\eqref{thm_amen_eq_subgrowth} hold, we can apply Lemma~\ref{lm_estimate} to semimetrics~$\rho_i$, weights $s_i = \abs{S^i_{n_j}}$, and $\phi = \phi(n_j)$. We obtain the following estimate:
  	%При больших значениях $j$ выполняются неравенства \eqref{thm_amen_eq_subgrowth}. Следовательно, лемма \ref{lm_estimate} применима для полуметрик~$\rho_i$, весов $s_i = \abs{S^i_{n_j}}$ и $\phi = \phi(n_j)$. Получим следующую оценку:
  	\begin{equation}
  		\mbbH_{\eps^4}(X^{\faktor{G}{H}},\mu^{\faktor{G}{H}}, G_{av}^{W_{n_j}}\rho) \ge \frac{1}{\phi(n_j)}\eps^3  \suml_{{i \in I_{n_j}}} \mbbH_{4\eps}(X, \mu, H_{av}^{S^i_{n_j}}{\tilde\rho}) - k_{n_j} - 1 \ge
  		 \frac{1}{\phi(n_j)^2}\eps^3 \abs{W_{n_j}} - 2k_{n_j}.
  	\end{equation}
  	%where $k_{n_j}$ is the number of non--empty $S^i_{n_j}$. 
  	Clearly, for any sequence $\psi(n)$, which goes to infinity, there is an increasing  $\phi(n)$ also going to infinity with  $\phi^2(n) = o(\psi(n))$. Note that $k_{n} = o(\abs{W_{n}})$, so we can choose such $\phi$ with $k_n = o({\phi(n)^{-2}}\abs{W_{n}})$. Therefore, the action $\alpha$ constructed by such slow growing sequence $\phi$ satisfies the following condition. For any $\Phi \in \mathcal{H}(\alpha, \theta)$ and sufficiently small~$\eps > 0$,
  	%где $k_{n_j}$ --- количество непустых $S^i_{n_j}$. Очевидно, для любой последовательности $\psi(n)$ растущей к бесконечности существует такая неубывающая $\phi(n)$, растущая к бесконечности, что $\phi^2(n) = o(\psi(n))$. Отметим, что $k_{n} = o(\abs{W_{n}})$, поэтому последовательность $\phi$ можно выбрать настолько медленной, что $k_n = o({\phi(n)^{-2}}\abs{W_{n}})$. 	Тогда, для действия $\alpha$ построенного по последовательности $\phi$ и любой $\Phi \in \mathcal{H}(\alpha, \theta)$ при достаточно малых $\eps > 0$
  	\[
  	 \Phi(n,\eps) \not \lesssim \frac{\abs{W_n}}{\psi(n)}.
  	\]
  	
  	Let us show that $\alpha$ has zero entropy. Indeed, it can be shown directly through a similar technique that every action coinduced from an action with zero entropy has zero entropy as well. The right inequality in~\eqref{lm_exampe_st} combined with Lemma~\ref{lm_upperbound}  shows that $\Phi_\rho(n,\eps) = o({|{W_n}|})$. Although $\rho$ may not be generating with respect to~$\theta$, we can choose $\theta^\p = \{W_n^\p\}$ with $|W_n^\p\triangle W_n| = o(|W_n|)$ such that $\rho$ is generating with respect to $\theta^\p$. Hence,   $\Phi_\rho(n,\eps) \in \mathcal{H}(\alpha, \theta^\p) = \mathcal{H}(\alpha, \theta)$ and the  amenable measure entropy of $\alpha$ is zero due to Theorem~\ref{thm_kolm}.
  	%Осталось лишь заметить, что действие, коиндуцированное с действия нулевой энтропии, также имеет нулевую энтропию.
  	
  	{
  	It only remains to prove that the constructed actions are ergodic. It is true if the index of $H$ is infinite. In this case, the ergodicity follows from the similar argument as in the case of Bernoulli shift. Suppose, that there is a non--trivial invariant subset~$E\subset \prod_{i=1}^\infty X_i$. This subset can be approximated by a cylinder set $C$ with $\mu(E \triangle C) < \eps$, where $\eps$ is arbitrarily small. Evidently, for any~$g \in G$ the set $g^{-1} C$ is cylinder as well and $\mu(g^{-1} C \triangle C) < 2\eps$. However, for any cylinder~$C$ we can find an element~$g \in G$ such that the basements of $C$ and $g^{-1}C$ do not intersect and, therefore,  $\mu(g^{-1} C \cap C) = \mu(C)^2 < \mu(C) - 2\eps$.
  	
  	In the case of finite index, the coinduced action itself may not be ergodic. However, we can consider its ergodic component~$\nu$ whose projections coincide with the initial measures~$\mu_i$. Since $\alpha$ has zero entropy, almost all such components have zero entropy as well. To estimate the scaling entropy of this system, we will use the following simple argument instead of Lemma~\ref{lm_estimate}. {Let $L$ be the index of $H$}. Then there exists some $S^{i_0}_{n_j}$ with $|S^{i_0}_{n_j}| \ge \frac{1}{L}|W_{n_j}|$. Therefore,
  	\begin{multline}
  	    \mbbH_{\eps}(X^{\faktor{G}{H}},\nu, G_{av}^{W_{n_j}}\rho) \ge \mbbH_{\eps}(X^{\faktor{G}{H}},\nu, L^{-1}\rho_{i_0}) \ge \\ \mbbH_{L\eps}(X^{\faktor{G}{H}},\nu, \rho_{i_0}) =
  	    \mbbH_{L\eps}(X_{i_0},\mu_{i_0}, \rho_{i_0}) > \frac{1}{\phi(n_j)}\abs{S^{i_0}_{n_j}} \ge \frac{1}{L\phi(n_j)} \abs{W_{n_j}}.
  	\end{multline}
  	Thus, in both cases the desired actions are constructed.} 
\end{proof}

\section{Adic action on the graph of ordered pairs}
%\section{Адическое преобразование на графе упорядоченных пар}
\label{S6}

	\newcommand{\point}{\mathfrak o}
	\newcommand{\word}{\mathfrak b}
	\newcommand{\edgemark}{\mathfrak c}
\subsection{Graph of ordered pairs}		
	    In order to construct {ergodic} actions of almost complete growth, we use the notion of \emph{the adic transformation} (Vershik's automorphism) on \emph{the graph of ordered pairs}. This graph was studied in detail in~\cite{VZ} and~\cite{Z2}.
		%Для построения систем полного роста  мы будем использовать конструкцию адического преобразования (автоморфизм Вершика) на  графе упорядоченных пар. Этот граф был подробно изучен в работах~\cite{VZ} и~\cite{Z2}.
		
		Consider an infinite graded graph $\Gamma = (V, E)$. The set of vertices $V$ is a disjoint union of the levels $V_n = \{0,1\}^{2^n}$, $n \geq 0$. The set of edges is defined together with the coloring $\edgemark \colon E \to \{0,1\}$ in the following way. Let $v_n \in V_n$ and $v_{n+1}\in V_{n+1}$. The edge $e = (v_n, v_{n+1})$ belongs to $E$ if and only~if $v_n$ is a prefix or a suffix of $v_{n+1}$. We mark this edge (define $\edgemark(e)$)  with $0$ or $1$ respectively. If~$v_n$ simultaneously forms both a prefix and a suffix of $v_{n+1}$, then we draw two distinct edges between $v_n$ and $v_{n+1}$ also marked with $0$ and $1$ respectively. The vertices $v_n$ and $v_{n+1}$ are called the initial and terminal points of $e$. We denote them by $s(e)$ and $r(e)$ respectively. 
		%Рассмотрим бесконечный градуированный граф $\Gamma = (V, E)$. Множество вершин $V$ графа $\Gamma$ есть дизъюнктное объединение множеств $V_n = \{0,1\}^{2^n}$ {, $n \geq 0$}. Множество рёбер $E$ определяется одновременно с раскраской $\edgemark \colon E \to \{0,1\}$  следующим образом. Пусть $v_n \in V_n$ и $v_{n+1}\in V_{n+1}$. Ребро $e = (v_n, v_{n+1})$ принадлежит $E$, если слово $v_n$ является началом или концом слова $v_{n+1}$ и помечено символом $0$ или $1$ соответственно. Если $v_n$ является одновременно началом и концом $v_{n+1}$, то в графе $\Gamma$ проводятся два ребра между $v_n$ и $v_{n+1}$, соответствующие цветам $0$ и $1$. 	Вершины $v_n$ и $v_{n+1}$ мы будем назвать началом и концом ребра $e$ соответственно, и обозначать символами $s(e)$ и $r(e)$.
		
		A path in $\Gamma$ is a sequence of edges $\{e_i\}$  such that $s(e_{i+1}) = r(e_i)$ and $s(e_i) \in V_i$. On the set $X$ of all infinite paths the cylinder topology is imposed in a natural way. A Borel measure on $X$ is called \emph{central} if all possible  beginnings of a path have equal probabilities while the tail is fixed. It means that any two cylinder sets whose corresponding finite paths have the same terminal vertex have the same measure.  
		%Путь в графе $\Gamma$ есть такая последовательность рёбер $\{e_i\}$, что $s(e_{i+1}) = r(e_i)$ и $s(e_i) \in V_i$. На множестве $X$ всех бесконечных путей естественным образом вводится цилиндрическая топология. Борелевская мера на пространстве $X$ называется центральной, если при фиксированном хвосте пути все его начала равновероятны, то есть, любые два цилиндрических множества, порождающие конечные пути которых заканчиваются в одной вершине, имеют равную меру.    
		
		Define \emph{the adic transformation~$T$} on the path space~$X$. Let $x = \{e_i\}_{i=0}^\infty$ be an infinite path. Find the minimal~$n$ with~$\edgemark(e_n) = 0$. Transformation $T$ maps $x$ to another path $T(x) = \{u_i\}$ defined in the following way. Let  $u_i = e_i$ for  $i \ge n+1$, $\edgemark(u_n) = 1$, and $\edgemark(u_i) = 0$ for all $i < n$ (see Fig.~\ref{fig:adic}). If~$\mu$ is a central measure, then this transformation is defined on a subset of full measure and forms an automorphism of the measure space  $(X, \mu)$.   
		%Определим \emph{адическое преобразование} $T$ на пространстве путей $X$. Пусть $x = \{e_i\}_{i=1}^\infty$ --- некоторый бесконечный путь. Найдём наименьшее такое $n$, что $\edgemark(e_i) = 0$. Определим путь $T(x) = \{u_i\}$ следующим образом. При $i \ge n+1$ выполнено $u_i = e_i$; $\edgemark(u_n) = 1$, и $\edgemark(u_i) = 0$ для всех $i < n$ (см. рис. \ref{fig:adic}). Относительно любой центральной меры $\mu$ преобразование $T$ является автоморфизмом пространства $(X, \mu)$.
		
		Now let $\sigma = \{\sigma_n\}$ be a given sequence of zeroes and ones. Let us construct a special central measure $\mu^\sigma$ on $X$. Note that any Borel measure $\mu$ on the path space is uniquely determined by a coherent system $\{\mu_n\}$, where each $\mu_n$ is a measure on a space $X_n$ of finite paths of length $n$. In terms of $\mu_n$, the centrality of $\mu$ means that for any $n$ the measure $\mu_n$ only depends on the terminal vertex of a path. Define then a measure $\nu_n$ on the $n$-th level $V_n$ as follows:
		%Зафиксируем некоторую последовательность $\sigma = \{\sigma_n\}$ состоящую из нулей и единиц. Построим соответствующую ей центральную меру $\mu^\sigma$ на пространстве $X$. Борелевская мера $\mu$ на пространстве $X$ однозначно определяется согласованной системой мер $\mu_n$ на цилиндрических множествах, соответствующих конечным путям длины $n$. В терминах $\mu_n$ центральность меры $\mu$ означает, что для любого $n$ мера $\mu_n$ зависит лишь от конца пути. Пусть $X_n$ есть множество всех конечных путей длины $n$. %{, начинающихся на нулевом уровне}. Определим меру $\nu_n$ на $V_n$ следующим образом:
		\[
		\nu_n(v) = \suml_{\substack{x \in X_n,\\ r(x) = v}} \mu_n(x).
		\]     
		The coherent system of measures $\{\nu_n\}$ uniquely determines the central measure  $\mu$.
		%Согласованная система мер $\nu_n$ однозначно определяет центральную меру $\mu$.
		
		Let us construct a sequence of finite subsets $V_n^\sigma \subset V_n$. Let $V_0^\sigma = V_0$. For $n \ge 1$, put $V_n^\sigma = \{ab \colon a,b \in V_{n-1}^\sigma\}$ if $\sigma_n = 1$ and $V_n^\sigma = \{aa \colon a \in V_{n-1}^\sigma\}$ otherwise. Let $\nu_n^\sigma$ be the uniform measure on the finite set $V_n^\sigma \subset V_n$. It is easy to see that $\{\nu_n^\sigma\}$ forms a coherent system. Let $\mu^\sigma$ be  the unique central measure on $X$ with the given coherent system~$\{\nu_n^\sigma\}$ (see~\cite{Z2} for details). 
		%Построим последовательность множеств $V_n^\sigma$, где $V_n^\sigma \in V_n$. Положим $V_0^\sigma = V_0$. При $n \ge 1$ если $\sigma_n = 1$ определим $V_n^\sigma = \{ab \colon a,b \in V_{n-1}^\sigma\}$; если $\sigma_n = 0$ положим $V_n^\sigma = \{aa \colon a \in V_{n-1}^\sigma\}$. Обозначим символом $\nu_n^\sigma$ равномерную меру на множестве $V_n^\sigma \subset V_n$. Построенная по этой системе мера $\mu^\sigma$ определена корректно и является центральной (см. \cite{Z2}).
		
		\begin{figure}[h]
		    \centering
		    \includegraphics[scale = 0.5]{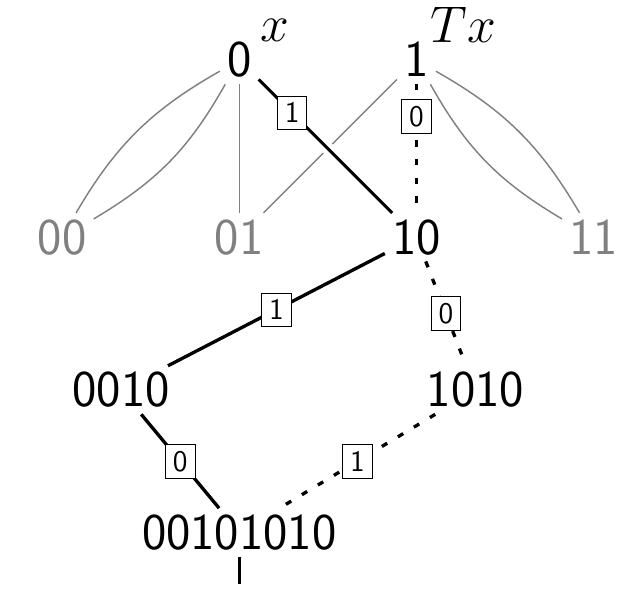}
		    \caption{The adic transformation. 
		    %Here $\word_3(Tx) = \word_3(x) = 00101010$, $\point_3(x) = 4$,  and $\point_3(Tx) = 5$.
		    }
		    \label{fig:adic}
		\end{figure}
		
		In~\cite{Z2}, it is proved that the system  $(X, \mu^\sigma, T)$ is stable \emph{with respect to the standard equipment} of the group~$\mathbb{Z}$. Moreover, it is shown that the sequence $h_n = 2^{s^\sigma(\log n)}$, where $s^{\sigma}(t) = \sum_{i < t} \sigma_i$, is a scaling entropy sequence of that system. In addition, for every $\sigma$ with an infinite number of ones, the transformation $T$ is ergodic. The Kolmogorov--Sinai entropy of $T$ is positive if and only if there are only finitely many zeroes in $\sigma$. Lemma~\ref{lm_example} of this work deals with a more complicated system of sets over which we take an averaging. However, we can restrict ourselves to establishing only lower bounds for epsilon--entropy.   
		%В работе \cite{Z2} доказано, что \emph{относительно стандартного оснащения} {группы $\mathbb{Z}$} система $(X, \mu^\sigma, T)$ является стабильной, и последовательность $h_n = 2^{s^\sigma(\log n)}$, где $s^{\sigma}(t) = \sum_{i < t} \sigma_i$, является масштабирующей последовательностью этой системы. Также, для любой последовательности $\sigma$, содержащей бесконечное число единиц преобразование $T$ является эргодическим, его энтропия положительна в том и только том случае, когда в $\sigma$ есть лишь конечное число нулей. Лемма \ref{lm_example} настоящей работы имеет дело с более сложной системой множеств, по которым производится усреднение. Однако, нам достаточно ограничится оценками $\eps$--энтропии снизу.
		
		Let $x = \{e_i\} \in X$ be an infinite path. We denote by~$\word_n(x)$ the vertex of $n$-th level which lies on~$x$. By $\point_n(x)$ we denote the value of  $\sum_{i = 0}^{n-1} \edgemark(e_i)2^i$. It easy to see that if $\point_n(x) < 2^n-1$ then $\word_n(Tx) = \word_n(x)$ and $\point_n(Tx) = \point_n(x) + 1$.
		%Пусть $x = \{e_i\} \in X$ --- некоторый бесконечный путь. Символом $\word_n(x)$ мы будем обозначать его вершину с номером $n$. А символом $\point_n(x)$ обозначим величину $\sum_{i = 0}^{n-1} \edgemark(e_i)2^i$. Легко видеть, что при $\point_n(x) < 2^n-1$ выполняется равенство $\word_n(Tx) = \word_n(x)$ и $\point_n(Tx) = \point_n(x) + 1$. 
	
	    Thus, we have described the construction of the graph of ordered pairs and the adic transformation on its path space. Now we will use this construction to prove Lemma~\ref{lm_example}.
	    %Let us proceed to the proof of Lemma~\ref{lm_example}.
	    %Итак, описание конструкции графа упорядоченных пар и адического преобразования на нем закончено. Перейдем теперь непосредственно к доказательству леммы \ref{lm_example}.

\subsection{Proof of Lemma~\ref{lm_example}}
%\subsection{Доказательство леммы \ref{lm_example}}
	
	In order to prove Lemma~\ref{lm_example}, we construct special measures $\mu^\sigma$ on the path space $X$. The adic transformation~$T$ on the space $(X,\mu^\sigma)$ produces the desired automorphism. The idea is to choose an appropriate $\sigma$ in which zeroes occur rarely. We will determine the positions of these zeroes inductively one by one. Note that for any $\sigma$ with an infinite number of zeroes, the adic transformation has entropy zero.  Therefore, the right hand side of inequality~\eqref{lm_exampe_st} holds automatically. Indeed, it follows from Theorem~\ref{thm_kolm} and the fact that the sequence formed by all  $S_n^i$ satisfies the F\o lner condition.   
	%Мы будем строить последовательность $\sigma$, нули в которой встречаются очень редко. Позиции, на которых они находятся, мы будем определять индуктивно. Отметим, что для любой $\sigma$ с бесконечным числом нулей метрическая энтропия адического преобразования равна нулю, поэтому правое неравенство \eqref{lm_exampe_st} выполнено автоматически. Действительно, ведь последовательность, составленная из всех множеств $S_n^i$, удовлетворяет условию Фёльнера. 
	
	It is sufficient to prove the left part of inequality~\eqref{lm_exampe_st} for an arbitrary admissible summable semimetric (which may not be generating). Indeed, the simple argument, which we have already used, shows that if it holds for some semimetric, then it holds for any generating one as well. Now let $\rho$ be a cut semimetric corresponding to a partition that separates paths according to their first vertices. It is also enough to prove only non--strict inequality (with the sign $\lesssim$ instead of $\lnsim$). Indeed, we can just change the sequence $\phi(n)$ to $\phi^{\frac{1}{2}}(n)$, for example.      
	%Левую часть неравенства \eqref{lm_exampe_st} достаточно проверять на любой (не обязательно порождающей) полуметрике.  Действительно, если оно выполнено для какой-то полуметрики, то выполнено и для любой порождающей. Пусть $\rho$ есть разрезная полуметрика, соответствующая двуэлементному разбиению, различающему пути по первой вершине. 
	
	Without loss of generality, we can assume that all the sets $S_n^i$ consist of positive numbers. Let us fix some positive $\eps < \frac{1}{10}$. Suppose that we have already chosen $p$ numbers  $q_1, \ldots, q_p$ and another $p$ numbers $n_1,\ldots, n_p$ such that 
	\[
	\frac{\abs{S_{n_j}^i}}{\phi(n_j)} < \mbbH_\eps\lr{X, \mu, T_{av}^{S_{n_j}^i} \rho},\qquad i = 1, \ldots, k_{n_j},
	\]
	holds for $j = 1, \ldots, p$ for any $\sigma$ whose first zeroes are exactly $q_1, \ldots, q_p$. Initially, we take~$p = 0$.
	%Можно считать, что все множества $S_n^i$ состоят из положительных чисел. Предположим, что уже выбрано $p$ чисел --- $q_1, \ldots, q_p$, и числа $n_1,\ldots, n_p$, такие, что левая часть неравенства \eqref{lm_exampe_st} выполнена при $j = 1, \ldots, p$ для любой $\sigma$, имеющей среди первых $q_p$ позиций ровно $p$ нулей  $q_1, \ldots, q_p$. На первом шаге положим $p = 0$.
	
	For $l > n_p$, we can assume that all the sets $\{S_l^i\}_{i=1}^{k_l}$ lie in the interval $\{0, \ldots, 2^{n(l)} - 1\}$. Let $N(l)$ be a sufficiently large constant. For example, we can put  $N = n+2$. 
	%to be such that inequality~\eqref{lm_example_eq_appr} below holds (for every $\sigma$ which has only $p$ zeroes $q_1, \ldots, q_p$ among its first $N$ symbols).
	For a binary word $v \in V_N$, we will denote its $k$-th bit by $v_k$. Note that for $x, y \in X$, the equality $\rho(x, y) = 0$ holds if and only if  $\word_N(x)_{\point_N(x)} = \word_N(y)_{\point_N(y)}$. Thus,
	%Для всякого $l > n_p$ найдём такое $n = n(l)$, что все множества $\{S_l^i\}_{i=1}^{k_l}$ лежат в интервале $\{0, \ldots, 2^n - 1\}$. Пусть $N(l)$ достаточно велико (а именно, положим $N$ настолько большим, чтобы выполнялось неравенство \eqref{lm_example_eq_appr}). Пусть $v \in V_N$ и $0 \le k \le 2^N - 1$. Символом $v_k$ обозначим бит, находящийся в слове $v$ в позиции $k$. Отметим, что равенство $\rho(x, y) = 0 $  эквивалентно $\word_N(x)_{\point_N(x)} = \word_N(y)_{\point_N(y)}$. Таким образом,
	\[
	T_{av}^{S_{l}^i} \rho (x, y) = \frac{1}{\abs{S_{l}^i}} \suml_{j \in S_{l}^i} \rho(T^j x, T^j y) = 
	\frac{1}{\abs{S_{l}^i}} \abs{\{j \in S_{l}^i \colon  \word_N(T^jx)_{\point_N(T^jx)} \not = \word_N(T^jy)_{\point_N(T^jy)}\}}.
	\] 
	
	Consider the set $A^{S_{l}^i} = \{0,1\}^{S_{l}^i}$ and the measure $\mu^\sigma_{S_{l}^i}$ on it defined as follows:
	%На множестве $A^{S_{l}^i} = \{0,1\}^{S_{l}^i}$ определим меру $\mu^\sigma_{S_{l}^i}$ следующим образом
	\[
	\mu^\sigma_{S_{l}^i}(w) = \mu^\sigma (x \in X\colon \word_N(T^jx)_{\point_N(T^jx)} = w_j, \ j \in  S_{l}^i).
	\]
	The mapping  
	$$\Phi\colon x \mapsto \lr{\word_N(T^jx)_{\point_N(T^jx)} }_{j \in S_{l}^i} $$ 
	produces an isomorphism of semimetric triples $(X, \mu^\sigma, T_{av}^{S_{l}^i} \rho)$ and $(A^{S_{l}^i}, \mu^\sigma_{S_{l}^i}, \rho^H)$, where~$\rho^H$ is a  Hamming distance on~$A^{S_{l}^i}$. The next step is to construct an appropriate uniform approximation of~$\mu^\sigma_{S_{l}^i}$. Note that for $\point_N(x) < 2^N - 2^n$ we have
	%задаёт изоморфизм полуметрических троек $(X, \mu^\sigma, T_{av}^{S_{l}^i} \rho)$ и $(A^{S_{l}^i}, \mu^\sigma_{S_{l}^i}, \rho^H)$, где $\rho^H$ есть метрика Хэмминга на  $A^{S_{l}^i}$. Заметим, что при $\point_N(x) < 2^N - 2^n$ 
	\[
	\Phi(x) = \lr{\word_N(x)_{\point_N(x)+j}}_{j \in S_{l}^i}.
	\]
	Due to centrality of $\mu^\sigma$, we have for any $N > n$
	%В силу центральности меры $\mu^\sigma$ для любого $N > n$ выполнено 
	\[
	\mu^\sigma(x \in X \colon \point_N(x) \ge 2^N - 2^n) = 2^{n-N}.
	\]
	Therefore, if $l$ is fixed and $N$ is large, then the measure $\mu^\sigma_{S_{l}^i}$ can be approximated pointwise simultaneously for all $i$ by the following measure 
	%Следовательно, при фиксированном $l$ мера $\mu^\sigma_{S_{l}^i}$ аппроксимируется поточечно при больших $N$ одновременно для всех $i$ мерой  
	\begin{multline}\label{lm_example_eq1}
	\mu^\sigma_{S_{l}^i, N}(w) = \frac{1}{1 - 2^{n - N}} \mu^\sigma(x \colon \Phi(x) = w, \ \point_N(x) < 2^N - 2^n) =\\
	\frac{1}{2^N - 2^n} \suml_{k = 0}^{2^N - 2^n -1} \nu_N^\sigma (v \in V_N\colon v_{k + j} = w_j, \ j \in  S_{l}^i).
	\end{multline}
	The last equality follows from centrality of $\mu^\sigma$. Note that for $N > n+1$, we have the inequality $\mu^\sigma_{S_{l}^i, N} < 2\mu^\sigma_{S_{l}^i}$ everywhere on  $A^{S_{l}^i}$. Therefore,
	%Выбирая $N(l)$ так, чтобы $\mu^\sigma_{S_{l}^i, N} < 2\mu^\sigma_{S_{l}^i}$, получим оценку 
	\begin{equation}\label{lm_example_eq_appr}
	\mbbH_\eps(A^{S_{l}^i}, \mu^\sigma_{S_{l}^i}, \rho^H) \ge \mbbH_{2\eps}(A^{S_{l}^i}, \mu^\sigma_{S_{l}^i, N}, \rho^H).
	\end{equation}
	Suppose that $q_{p+1}$, which is not defined yet,  is greater than $N$. Then any two summands in the last sum in~\eqref{lm_example_eq1} whose difference in indices is a multiple of $2^{q_p}$ coincide. It follows from the construction of $\nu_N^\sigma$. Therefore,
	%Мы будем считать, что ещё не выбранное  $q_{p+1}$  будет больше $N$. Тогда, при $k$ отличающихся на $2^{q_p}$, слагаемые в выражении \eqref{lm_example_eq1} совпадают, в силу конструкции меры $\nu_N^\sigma$.  Поэтому, при $n > q_p$
	\[
	\mu^\sigma_{S_{l}^i, N}(w) = \frac{1}{2^{q_p}} \suml_{k = 0}^{2^{q_p} -1} \nu_N^\sigma (v \in V_N\colon v_{k + j} = w_j, \ j \in  S_{l}^i).
	\]
	However, the epsilon--entropy of a semimetric with respect to a convex combination of measures can be estimated from below by the epsilon--entropy of this semimetric with respect to one of these measures. Therefore, it suffices to provide a lower bound for the $2\eps$--entropy of  $(A^{S_{l}^i}, \mu^\sigma_{S_{l}^i, N, k}, \rho^H)$ for $k = 0, \ldots , 2^{q_p} -1$, where 
	%Но $\eps$--энтропия полуметрики относительно выпуклой комбинации мер не меньше, чем $\eps$--эн\-тро\-пи\-я этой полуметрики относительно какой-то из усредняемых мер. Поэтому, достаточно оценить снизу $\eps$--энтропию $(A^{S_{l}^i}, \mu^\sigma_{S_{l}^i, N, k}, \rho^H)$ при $k = 0, \ldots , 2^{q_p} -1$, где  
	\[
	\mu^\sigma_{S_{l}^i, N, k}(w) = \nu_N^\sigma (v \in V_N\colon v_{k + j} = w_j, \ j \in  S_{l}^i).
	\]
	{Thus, it only remains to show that
	\begin{equation}\label{lm_example_eq2768}
	    \mbbH_{2\eps} (A^{S_{l}^i}, \mu^\sigma_{S_{l}^i, N, k}, \rho^H) > \frac{\abs{S_{l}^i}}{\phi(l)},
	\end{equation}
	for sufficiently large $l$, for any $i$, and for any $k = 0, \ldots , 2^{q_p} -1$.}
	
	Note that all components of $w$ are divided into groups in such a way that all coordinates in one group are the same, and distinct groups are independent with respect to $\mu^\sigma_{S_{l}^i, N, k}$. Indeed, that is true for $\nu_N$, and measure $\mu^\sigma_{S_{l}^i, N, k}$ is a projection of $\nu_N$ onto some chosen coordinates.  
	%Относительно $ \mu^\sigma_{S_{l}^i, N, k}$  все координаты $w$ разбиваются на группы равных, причём разные группы независимы. Действительно, мера $\nu_N$ на двоичных словах длины $2^N$ такова, а $\mu^\sigma_{S_{l}^i, N, k}$ получается из неё проекцией на пространство, порождённое несколькими выбранными координатами. 
	
	Now we will use that all  $S_l^i$ are $\frac{1}{l}$--invariant sets. Each $S_l^i$ consists of some intervals of integer numbers. The $\frac{1}{l}$--invariance of $S_l^i$ means that the number of these intervals does not exceed  $\frac{1}{l}\abs{S_l^i}$. The number of different groups that have points both inside and outside of a given interval does not exceed $2^{q_p+1}$ because the length of each group is not greater than $2^{q_p}$. Then the total number of points in such groups does not exceed 
	%При больших $l$ все множества $S_l^i$ являются $\frac{1}{l}$--инвариантными. Рассмотрим интервалы, составляющие некоторое $S_l^i$. Их количество не больше, чем $\frac{1}{l}\abs{S_l^i}$. Число групп, пересекающихся с данным интервалом, но не лежащих в нем целиком, не превосходит $2^{q_p+1}$, так как длина каждой группы не больше $2^{q_p}$. Тогда суммарный размер таких групп не превосходит 
	\begin{equation}\label{bigl}
	\frac{2^{q_p+1}}{l}\abs{S_l^i} < l^{-\frac{1}{2}} \abs{S_l^i}.
	\end{equation}
	The inequality holds for $l > 2^{2q_p+2}$. We will call a component \emph{proper} if it does not lie in the union of such groups. Let $\tilde\rho^H$ be a Hamming metric on the proper coordinates. For $l > 4$, we have
	%при достаточно большом $l$. Рассмотрим $\tilde\rho^H$ --- метрику Хемминга на координатах, составляющие группы, лежащие целиком в $S_l^i$. При больших $l$ выполнено неравенство 
	\[
	\tilde\rho^H(x,y) \le \frac{1}{1 -l^{{-}\frac{1}{2}}} \rho^H (x,y) \le 2 \rho^H (x,y) \quad x,y \in A^{S_l^i}. 
	\] 
	Then
	%Тогда 
	\begin{equation}\label{lm_example_eq_uniforminv}
	\mbbH_{2\eps}(A^{S_{l}^i}, \mu^\sigma_{S_{l}^i, N, k}, \rho^H) \ge 
	\mbbH_{4\eps}(A^{S_{l}^i}, \mu^\sigma_{S_{l}^i, N,k}, \tilde\rho^H).
	\end{equation}
	The right hand side of formula~\eqref{lm_example_eq_uniforminv} is exactly $4\eps$--entropy of a binary cube whose dimension is at least ${\abs{S_{l}^i}}{2^{-p-1}}$. This value is not less than 
	%Правая часть формулы \eqref{lm_example_eq_uniforminv} есть $4\eps$--энтропия двоичного куба размерности хотя бы ${\abs{S_{l}^i}}{2^{-p-1}}$, что, в свою очередь, не меньше
	\begin{equation}\label{lm_example_cond1}
	c(\eps)\frac{\abs{S_{l}^i}}{2^{p+1}} > \frac{\abs{S_{l}^i}}{\phi(l)}
	\end{equation} 
	for sufficiently large $l$. It only remains to choose $n_{p+1} = l$ which satisfies condition~\eqref{lm_example_cond1} and the corresponding $N$. Then put $q_{p+1} = N+1$.
	%Осталось только выбрать такое $l = n_{p+1}$, удовлетворяющее соотношению \eqref{lm_example_cond1}, и соответствующее ему $N$. После чего положим  $q_{p+1} = N+1$. 

\subsection{Remarks on the sequential entropy} 
In this section, we discuss a supplemental result that follows from the arguments similar to those which we provide above. {It turns out that the estimates for the scaling entropy of the adic action also can be applied to the sequential entropy~\cite{Kush}.
\begin{defn}
Given a sequence $A = \{a_n\}_{n=1}^\infty$ of positive integers. Let $R$ be a measure--preserving transformation of a measure space $(Y, \nu)$ and $\xi$ be a finite partition. Set 
\[
h_A(R, \xi) = \limsup_{n\to \infty} \frac{1}{n} H(R^{-a_1}\xi \vee \ldots \vee R^{-a_n}\xi),
\]
\[ 
h_A(R) = \sup_{\xi} h_A(R, \xi),
\]
where the supremum is taken over all finite partitions of $(Y, \nu)$.
\end{defn}}
\begin{thm}\label{thm_kush}
Let $A = \{a_n\}_{n=1}^\infty$ be a sequence of positive integers such that $a_{n+1} - a_n \to \infty$. Then there exists $\sigma$ such that the adic action $(X, \mu^\sigma, T)$ satisfies 
%an ergodic transformation $T$ with zero entropy such that
\[
h_A(T) > 0,
\]
where $h_A(T)$ is the sequential entropy of $T$ corresponding to $A$. 
\end{thm}

{Theorem~\ref{thm_kush} answers the question about the existence of a zero entropy system which has positive sequential entropy with respect to $A = \{n^2\}$ (see~\cite{KKW}, Question~6.4.3).}
\begin{proof}
    %Instead of considering a sequence of families of $\frac{1}{n}$--invariant sets $S_n^i$, we consider the equipment by the sets $S_n = \{a_1, \ldots, a_n\}$ with growing gaps. 
    Consider a sequence of finite sets $S_n = \{a_1, \ldots, a_n\}$. 
    The goal is to find a subsequence $n_j$ such that for any sufficiently small $\eps$
    \begin{equation}\label{thm_kush_eq1}
         \mbbH_\eps\lr{X, \mu, T_{av}^{S_{n_j}} \rho} \gtrsim n_j.   
    \end{equation}
    Let us repeat the proof of Lemma~\ref{lm_example} with $S_l$ instead of $S_l^i$ up to the sufficiency of the estimate {(i.\,e., formula \eqref{lm_example_eq2768})}
    \[
    \mbbH_{2\eps}(A^{S_{l}}, \mu^\sigma_{S_{l}, N, k}, \rho^H) \gtrsim l,
    \]
    since the $\frac{1}{l}$--invariance of $S_l^i$ appears only in the further arguments. Recall that with respect to $\mu^\sigma_{S_{l}, N, k}$ all the components are divided into groups of equal and distinct groups are independent. The length of each group does not exceed $2^{q_p}$, and since $a_{m+1} - a_m > 2^{q_p}$ starting with some $m_0$, we can find $l$ such that there are at least $\frac{1}{2}l$ groups which intersect $S_l$ by a unique component. Let~$\tilde\rho^H$ be a Hamming metric on these coordinates. Then we have $\tilde\rho^H \le  2 \rho^H$ as before. And finally, it suffices to estimate from below the $4\eps$--entropy of  $(A^{S_{l}}, \mu^\sigma_{S_{l}, N, k}, \tilde\rho^H)$ which is the $4\eps$--entropy of a binary cube of dimension at least $\frac{1}{2}l$ that is $c(\eps)l$. 
    
    Recall that $\rho$ is a cut semimetric corresponding to a finite partition $\xi$. Relation \eqref{thm_kush_eq1} combined with Lemma~\ref{lm_partitions} implies that $H(\bigvee_{i=0}^{n_j} T^{-a_i}\xi) \gtrsim n_j$ and thus $h_A(T) > 0$.
\end{proof}

\section{Proof of Lemma~\ref{lm_estimate}}
%\section{Доказательство леммы \ref{lm_estimate}} 
\label{sec_proof_est}
        
    Let us proceed to the last step, which is the proof of Lemma~\ref{lm_estimate}, that we need in order to complete the proof of the main theorem. First, we construct an appropriate partition of each measure space $(X_i, \mu_i)$ with semimetric $\rho_i$. These partitions provide a useful framework to deal with the epsilon--entropy of a product space. Second, we apply some probabilistic estimates that lead to the desired inequality.  
    %Теперь приступим к доказательству последнего шага в доказательстве основной теоремы, а именно леммы \ref{lm_estimate}. Сначала построим для каждого пространства $(X_i, \mu_i)$ специальное разбиение следующим образом. 
    
    For $i = 1,\ldots, k$, denote by $b_i$ the value of $2^{\mbbH_{4\eps}(X_i, \mu_i, \rho_i)}$. Since all the semimetrics $\rho_i$ are admissible, they have finite $\eps^2$--entropies. Consider the corresponding partition of $X_i$. Since $(X_i, \mu_i)$ is a continuous Lebesgue space, there exists a refinement $Y_0, \ldots, Y_r$ of this partition which satisfies the following: for~$j > 0$, we have  $\diam_\rho(Y_j) < \eps^2$, $\mu(Y_0) < 2\eps^2$, and $\mu(Y_{j_1}) = \mu(Y_{j_2})$ for all~$j_1, j_2 > 0$. 
    %Символом $b_i$ обозначим величину $2^{\mbbH_{4\eps}(X_i, \mu_i, \rho_i)}$. Так как полуметрика $\rho_i$ допустима, то её $\eps^2$-энтропия конечна. Рассмотрим соответствующее разбиение пространства $X_i$. Так как пространство стандартно, измельчим его до разбиения $Y_0, \ldots, Y_r$, для которого $\mu(Y_0) < 2\eps^2$, $\diam_\rho(Y_j) < \eps^2$ при $j > 0$, и $\mu(Y_{j_1}) = \mu(Y_{j_2})$ для всех $j_1, j_2 > 0$. 
		
	Consider the following procedure. Note that for any measurable $Z \subset X_i$ with $\mu(Z) < 4\eps$ there exists a $2\eps$--separated set of size $b_i$ in the difference  $X_i\setminus Z$. Put $Z_0 = Y_0$ and choose the corresponding  $2\eps$--separated set $\{p_1, \ldots, p_{b_i}\}$ in $X_i\setminus Z_0$. For each point $p_j$ find a cell $Y_j$ containing it and denote this sell by~$a^i_{1,j}$. Thus, we obtain a family $\{a^i_{1,j}\}_{j=1}^{b_i}$ of disjoint subsets. Let us denote the union of these subsets by $A^i_1$. Note that these sets satisfy the following property. For any  $x_i\in a^i_{1,j_1},\ y_i \in a^i_{1,j_2}$ with $ j_1 \not = j_2$ the distance between $x_i$ and $y_i$ is at least $2\eps - 2\eps^2 > \eps$ due to the triangle inequality. If $\mu(A^i_1) < \eps$, we can choose $Z_1 = Z_0 \cup A^i_1$ and similarly extract~$A^i_2$ from $X_i\setminus Z_1$ that is a union of subsets $a^i_{2,j}$, $j=1,\dots, b_i$, satisfying the same property. Thus, we can repeat this procedure until we obtain the following partition of $(X_i, \mu_i)$:
	%Для любого измеримого подмножества $Z \subset X_i$, такого, что $\mu(Z) < 4\eps$, в множестве $X_i\setminus Z$ существует $2\eps$--разделённое множество размера $b_i$. Рассмотрим $Z_0 = Y_0$ и выберем соответствующее $2\eps$--разделённое множество $\{p_1, \ldots, p_{b_i}\}$ в разности $X_i\setminus Z_0$. Для каждого $p_j$ найдём содержащее его множество $Y_j$ и обозначим его $a^i_{1,j}$. Таким образом, получим набор $\{a^i_{1,j}\}_{j=1}^{b_i}$ дизъюнктных множеств. Символом $A^i_1$ обозначим объединение этого набора. Заметим, что для любых $x_i\in a^i_{1,j_1},\ y_i \in a^i_{1,j_2}$ таких, что $ j_1 \not = j_2,$ расстояние между $x_i$ и $y_i$ в силу неравенства треугольника не меньше, чем $2\eps - 2\eps^2 > \eps$. Если $\mu(A^i_1) < \eps$, выберем $Z_1 = Z_0 \cup A^i_1$ и выделим аналогичным способом множество $A^i_2$ { в $X_i\setminus Z_1$, являющееся объединением подмножеств $a^i_{2,j}$, $j=1,\dots, b_i$}. Повторяя данную процедуру пока это возможно, получим следующее разбиение пространства $(X_i, \mu_i)$:
		\[
		X_i = \bigcup\limits_{l = 0}^{m_i} A_l^i, 
		\] 
		where $\mu_i(A_0^i) \le 1 - \eps$ and any $A_l^i$ with $l > 0$ admits a decomposition
		%где $\mu_i(A_0^i) \le 1 - \eps$, а каждое множество $A_l^i$ при $l > 0$ допускает подразбиение 
		\[
		A_l^i = \bigcup\limits_{j=1}^{b_i} a_{l, j}^i
		\]
		such that for any $x \in a_{l, j_1}^i$ and $y \in a_{l, j_2}^i$ the $\rho_i$--distance between  $x$ and $y$ is not less than $\eps$, and all sets $a_{l, j}^i,\ l=1, \ldots, m_i,\ j = 1, \ldots, b_i$, have the same measure. 
		%такое, что для любых $x \in a_{l, j_1}^i$ и $y \in a_{l, j_2}^i$ расстояние в полуметрике $\rho_i$ между $x$ и $y$ не меньше $\eps$, и все множества $a_{l, j}^i,\ l=1, \ldots, m_i,\ j = 1, \ldots, b_i$, имеют равную меру. 
	
		Now let us estimate the $\eps^4$--entropy of $(X,\mu,\rho)$ from below. Assume that a set $E \subset X$ with measure less than $\eps^4$ is given. We will look for a  $\eps^4$--separated set in its complement. For any point $x=(x_1,\dots,x_k) \in X$, we define a sequence $w = w(x) \in \prod_i \{1, \ldots, m_i\}$ of $k$ non--negative integers $\{w_r\}_{r = 1}^k$ in the following way:
		%Оценим снизу $\eps^4$--энтропию { пространства $(X,\mu,\rho)$}. Пусть множество $E \subset X$ меры меньше $\eps^4$ выбрано. Будем искать $\eps^4$--разделённое множество в его дополнении. Для каждой точки $x=(x_1,\dots,x_k) \in X$ определим последовательность $w = w(x) \in \prod_i \{1, \ldots, m_i\}$ из $k$ целых неотрицательных чисел $\{w_r\}_{r = 1}^k$ следующим образом:
		\[
		x_r \in A_{w_r}^r \text{ for } r = 1, \ldots, k.
		\] 
		
		Let us fix an arbitrary $w$ satisfying the inequality 
		%Зафиксируем такую последовательность $w$, что
		\begin{equation}\label{lm_estimate_eq_condition}
			\suml_{w_r \not = 0 } s_r \ge \eps^2 \suml_{i=1}^k s_i.
		\end{equation}
		Consider the following set $S_w = \{x \in X\colon w(x) = w\}$. It is easy to see that $S_w$ can be represented as the following disjoint union:
		%Рассмотрим множество $S_w = \{x \in X\colon w(x) = w\}$.  Ясно, что это множество представляется как следующее дизъюнктное объединение
		\begin{equation}\label{lm_estimate_eq_part}
			S_w = \bigcup\limits_{{i_r = 1, \ldots, b_r}} a^1_{w_1,i_1} \times a^2_{w_2, i_2} \times \dots \times a^k_{w_k,i_k},
		\end{equation}
		where we take the union over those indices  $i_r$ for which $w_r \not = 0$, and all the factors corresponding to $w_r = 0$ are equal to $A^r_0$. We will call the subsets on the right side of formula~\eqref{lm_estimate_eq_part} \emph{the cells} of~$S_w$. Note that all the cells have the same measure. For a point $x_i \in A_{l}^r$, we denote the set~$a^r_{l, j}$ containing~$x_i$ by~$a^r_{l}(x_i)$.
		%где объединение производится по тем индексам $i_r$ для которых $w_r \not = 0$, а в качестве множителя соответствующего $w_r = 0$ выступает $A^r_0$. Отметим, что все множества в объедении имеют равную меру. Множества из объединения в правой части формулы \eqref{lm_estimate_eq_part} мы будем называть клетками, составляющими $S_w$. Для точки $x_i \in A_{l}^r$ обозначим символом $a^r_{l}(x_i)$ множество $a^r_{l, j}$, содержащее $x_i$. 
		Let $x, y \in S_w$, then
		%Пусть $x, y \in S_w$, тогда
		\[
		\rho(x,y) \ge \frac{\eps}{\suml_{i=1}^k s_i} \suml_{w_r \not = 0} s_r\indicator\{{a^r_{w_r}(x_r) \not = a^r_{w_r}(y_r)}\}  \ge 
		\frac{\eps^3}{\suml_{w_r \not = 0} s_r} \suml_{w_r \not = 0} s_r\indicator\{{a^r_{w_r}(x_r) \not = a^r_{w_r}(y_r)}\}.
		\] 
		
		Assume that the subset $E$ contains less than a half of cells of~$S_w$ entirely. Let us estimate from below the cardinality of the maximal $\eps^4$--separated set in  $S_w \setminus E$. To do that, it suffices to establish an upper bound for the measure of an $\eps$--ball   on the space $\prod_{w_r \not = 0} \{1, \ldots, b_r\}$ with the uniform measure and the following semimetric~$\tilde\rho$:
		%Предположим, что в множество $E$ целиком попало менее половины клеток, составляющих~$S_w$. Оценим снизу размер максимального $\eps^4$--разделённого множества в $S_w \setminus E$. Для этого достаточно оценить сверху меру $\eps$--шара на пространстве $\prod_{w_r \not = 0} \{1, \ldots, b_r\}$ с равномерной мерой и метрикой $\tilde \rho$, заданной следующим образом:
		\[
		\tilde\rho (u,v) =  \frac{1}{\suml_{w_r \not = 0} s_r} \suml_{w_r \not = 0} s_r\indicator\{u_r \not = v_r\}.
		\]
		The random variables $u_r$ are mutually independent, and each $u_r$ is uniformly distributed on the set $\{1, \ldots, b_r\}$. Thus, it is enough to estimate the probability 
		%Случайные величины $u_r$ независимы и равномерно распределены на множествах $\{1, \ldots, b_r\}$.  Тогда достаточно оценить вероятность
		
		%\begin{multline}
		%	\mbbP\{ \suml_{w_r \not = 0} s_r \indicator\{ u_r \not = 1\} \le \eps \suml_{w_r \not = 0} s_r \} = 
		%	\mbbP\{ \suml_{w_r \not = 0} s_r \indicator\{ u_r = 1\} \ge (1 - \eps) \suml_{w_r \not = 0} s_r\} \le \\
		%	\mbbP\{ \suml_{w_r \not = 0} \log b_r \indicator\{ u_r = 1\} \ge \frac{1 - \eps}{\phi} \suml_{w_r \not = 0} s_r\} \le
		%	e^{- \frac{(1-\eps) \suml_{w_r \not = 0} s_r}{\phi}} \cdot \mbbE \lr{\prod\limits_{w_r \not = 0} e^{\log b_r \indicator\{u_r =1\}}}.
		%\end{multline} 
		
		\begin{multline}
			\mbbP\{ \suml_{w_r \not = 0} s_r \indicator\{ u_r \not = 1\} \le \eps \suml_{w_r \not = 0} s_r \} = 
			\mbbP\{ \suml_{w_r \not = 0} s_r \indicator\{ u_r = 1\} \ge (1 - \eps) \suml_{w_r \not = 0} s_r\} \le \\
			\mbbP\{ \suml_{w_r \not = 0} \log b_r \indicator\{ u_r = 1\} \ge \frac{1 - \eps}{\phi} \suml_{w_r \not = 0} s_r\} \le
			2^{- \frac{(1-\eps) \suml_{w_r \not = 0} s_r}{\phi}} \cdot \mbbE \lr{\prod\limits_{w_r \not = 0} 2^{\log b_r \indicator\{u_r =1\}}}.
		\end{multline} 
		The first inequality follows from the conditions for the weights. The second inequality holds due to the exponential Chebyshev's inequality. Let us estimate the first factor:   
		%Последний переход справедлив в силу экспоненциального неравенства Чебышева. Оценим первый множитель.
		%\begin{equation}
		%		e^{- \frac{(1-\eps) \suml_{w_r \not = 0} s_r}{\phi}} \le 	
		%		e^{- \frac{(1-\eps)\eps^2 \suml_{i = 1}^k s_i}{\phi}} \le \\ 
		%		e^{- \frac{(1-\eps)\eps^2 \suml_{i = 1}^k \log b_i}{\phi}} 
		%		= \lr{\prod_{i=1}^k b_i}^{-\frac{(1-\eps)\eps^2}{\phi}}
		%		{\le \lr{\prod_{i=1}^k b_i}^{-\frac{\eps^3}{\phi}}}.	
		%\end{equation}
		
		\begin{equation}
				2^{- \frac{(1-\eps) \suml_{w_r \not = 0} s_r}{\phi}} \le 	
				2^{- \frac{(1-\eps)\eps^2 \suml_{i = 1}^k s_i}{\phi}} \le \\ 
				2^{- \frac{(1-\eps)\eps^2 \suml_{i = 1}^k \log b_i}{\phi}} 
				= \lr{\prod_{i=1}^k b_i}^{-\frac{(1-\eps)\eps^2}{\phi}}
				{\le \lr{\prod_{i=1}^k b_i}^{-\frac{\eps^3}{\phi}}}.	
		\end{equation}
		To estimate the second factor we use the independence of $u_r$:
		%Then, we estimate the second factor using that $u_r$ are independent:
		%Оценим второй множитель, пользуясь независимостью входящих в него случайных величин:  
		\begin{equation}
			\prod\limits_{w_r \not = 0} \mbbE\lr{b_r^{\indicator\{u_r =1\}}} \le 
			\prod\limits_{i = 1}^k \lr{\frac{b_i}{b_i} + \frac{b_i -1}{b_i}} \le 
			 %\lr{\prod_{i=1}^k b_i}^{\frac{\eps^5}{\phi}},
			 2^k.
		\end{equation}
		
		Thus, the desired probability does not exceed  $\lr{\prod_{i=1}^k b_i}^{-\frac{\eps^3}{\phi}}2^k$. Therefore, the size of the maximal $\eps^4$--separated set in $S_w \setminus E$ can be estimated from below by the value $\frac{1}{2}\lr{\prod_{i=1}^k b_i}^{\frac{\eps^3}{\phi}} 2^{-k}$. Hence,
		%Итого, искомая вероятность не превосходит $ \lr{\prod_{i=1}^k b_i}^{-\frac{\eps^3}{\phi}}2^k$. Следовательно, размер максимального $\eps^4$--разделённого множества в $S_w \setminus E$ не меньше, чем  $\frac{1}{2}\lr{\prod_{i=1}^k b_i}^{\frac{\eps^3}{\phi}} 2^{-k}$. Таким образом,
		\[
		\mbbH_{\eps^4}(X,\mu, \rho) \ge
		\log \lr{\lr{\prod_{i=1}^k b_i}^{\frac{\eps^3}{\phi}} 2^{-k-1} }= 
		 \frac{1}{\phi}\eps^3  \suml_{i=1}^k \mbbH_{4\eps}(X_i, \mu_i, { \rho_{i}}) -k-1.
		\]
		If the desired inequality does not hold, then $E$ must entirely contain at least half of cells of~$S_w$ for~$w$ satisfying  condition~\eqref{lm_estimate_eq_condition}.
		%Стало быть, если утверждение леммы неверно, то множество $E$ содержит целиком хотя бы половину клеток каждого множества $S_w$, при  $w$, удовлетворяющей условию \eqref{lm_estimate_eq_condition}.
		Let us estimate the measure of those $x\in X$ that do not satisfy condition~\eqref{lm_estimate_eq_condition}:  
		%Оценим меру тех $x\in X$, для которых условие \eqref{lm_estimate_eq_condition} не выполнено: 
		{
		\begin{multline}
			\mu\Big\{x \in X \colon \suml_{i = 1}^{k} s_i \indicator\{x_i \in A_0^i\} \ge (1-\eps^2) \suml_{i = 1}^k s_i \Big\} \le\\ 
			\frac{\mbbE \suml_{i = 1}^{k} s_i \indicator\{x_i \in A_0^i\}}{ (1-\eps^2)\suml_{i = 1}^k s_i} = 
			\frac{\suml_{i = 1}^{k} s_i \mu_i(A_0^i)}{(1-\eps^2) \suml_{i = 1}^k s_i} \le 
			\frac{1-\eps}{1-\eps^2} \le 1 - \frac{\eps}{2}.			
		\end{multline}
		}
		Therefore, the measure of those $x \in X$ that satisfy condition~\eqref{lm_estimate_eq_condition} is at least~$\frac{\eps}{2}$. Hence, $\mu(E) \ge \frac{\eps}{4}$, and we obtain a contradiction to the choice of the exceptional set. The lemma is proved. 
		%Следовательно, те $x \in X$, для которых условие \eqref{lm_estimate_eq_condition} выполнено, составляют меру хотя бы $\frac{\eps}{2}$. Стало быть, $\mu(E) \ge \frac{\eps}{4}$, что противоречит выбору множества $E$.
%\end{proof}

\section*{Acknowledgements}
%\section*{Благодарности}
The author is sincerely grateful to his advisor Pavel~Zatitskiy for many helpful discussions. The author is also grateful to Valery~Ryzhikov for drawing the author's attention to this question. 
%Автор благодарен своему научному руководителю, Павлу Борисовичу Затицкому, за множество ценных обсуждений и помощь в работе. 

\bibliographystyle{amsplain}

\end{document}